\documentclass{article}%
\usepackage{amsfonts}
\usepackage{amsmath}
\usepackage{amssymb}
\usepackage{graphicx}%
\providecommand{\U}[1]{\protect\rule{.1in}{.1in}}
\DeclareMathOperator{\Gr}{Gr}
\DeclareMathOperator{\Stab}{Stab.Prod}
\DeclareMathOperator{\inv}{inv}
\DeclareMathOperator{\sign}{sign}
\newtheorem{theorem}{Theorem} [section]

\newtheorem{corollary}[theorem]{Corollary}

\newtheorem{definition}[theorem]{Definition}
\newtheorem{example}[theorem]{Example}

\newtheorem{lemma}[theorem]{Lemma}

\newtheorem{observation}[theorem]{Observation}

\newtheorem{proposition}[theorem]{Proposition}
\newtheorem{remark}[theorem]{Remark}

\newenvironment{proof}[1][Proof]{\noindent\textbf{#1.} }{\ \rule{0.5em}{0.5em}}
\begin{document}

\author{Robert Shwartz\\Department of Mathematics\\Ariel University, Israel\\robertsh@ariel.ac.il
\and Vadim E. Levit\\Department of Computer Science\\Ariel University, Israel\\levitv@ariel.ac.il}
\title{Signed Hultman Numbers and Signed Generalized Commuting Probability in Finite Groups}
\date{}
\maketitle

\begin{abstract}
\noindent
Let $G$ be a finite group. Let $\pi$ be a permutation from $S_{n}$. We study
the distribution of probabilities of equality
\[
\label{main.eq.1} a_{1}a_{2}\cdots a_{n-1}a_{n}=a_{\pi_{1}}^{\epsilon_{1}%
}a_{\pi_{2}}^{\epsilon_{2}}\cdots a_{\pi_{n-1}}^{\epsilon_{n-1}}a_{\pi_{n}%
}^{\epsilon_{n}},
\]
when $\pi$ varies over all the permutations in $S_{n}$, and $\epsilon_{i}$
varies over the set $\{+1, -1\}$. By \cite{CGLS}, the case where all
$\epsilon_{i}$ are $+1$ led to a close connection to Hultman numbers. In this
paper we generalize the results, permitting $\epsilon_{i}$ to be $-1$. We
describe the spectrum of the probabilities of signed permutation equalities in
a finite group $G$. This spectrum turns out to be closely related to the
partition of $2^{n}\cdot n!$ into a sum of the corresponding signed Hultman numbers.



\end{abstract}

\section{Introduction}

The study of the probability that two random elements in a finite group $G$
commute (\textit{commuting probability}) has received considerable research attention. In 1968, Erd\"{o}s
and Turan proved that%
\[
\Pr(a_{1}a_{2}=a_{2}a_{1})>\frac{\log(\log|G|))}{|G|}.
\]
In early 1970s, Dixon observed that the commuting probability is $\leq\frac
{1}{12}$ for every finite non-abelian simple group (this was submitted as a
problem in Canadian Mathematical Bulletin \textbf{13} (1970), with a solution
appearing in 1973). In 1973, Gustafson proved that the commuting probability is
equal to $\frac{k(G)}{|G|}$, where $k(G)$ is the number of conjugacy classes
in $G$ \cite{G}. Based on this observation, Gustafson then obtained the
upper bound of the commuting probability in any finite non-abelian group to be
$\frac{5}{8}$ \cite{G}. This upper bound is actually attained in many finite
groups, including the two non-abelian groups of order $8$. Since then, there
has been significant research concerning probabilistic aspects of finite
groups. Many of these studies can be regarded as variations of the commuting
probability problem. For example, Das and Nath \cite{NathDash1},
\cite{DasNath1}, \cite{DasNath2} study the probability ${\Pr}_{g}^{\omega}(G)$ of the equality
\[
a_{1}a_{2}...a_{n-1}a_{n}a_{\pi_{1}}^{-1}a_{\pi_{2}}^{-1}\cdots a_{\pi_{n-1}%
}^{-1}a_{\pi_{n}}^{-1}=g
\]
in a finite group $G$ . The word
\[
a_{1}a_{2}...a_{n-1}a_{n}a_{\pi_{1}}^{-1}a_{\pi_{2}}^{-1}\cdots a_{\pi_{n-1}%
}^{-1}a_{\pi_{n}}^{-1},
\]
in which $a_{1}a_{2}...a_{n-1}a_{n}$ vary over all the elements of $G$, is
denoted by $\omega$. Thus this is a generalization of the classical study of
the commuting probability, in which case $\omega=a_{1}a_{2}a_{2}^{-1}%
a_{1}^{-1}$ and $g=1$ \cite{Nath2011}. In the same direction of generalization
of the commuting probability, Cherniavsky, Goldstein, Levit, and Shwartz
\cite{CGLS} have introduced an interesting connection between the distribution
of the probabilities
\[
{\Pr}_{\pi}(G)=\Pr(a_{1}a_{2}\cdots a_{n}=a_{\pi_{1}}a_{\pi_{2}}\cdots
a_{\pi_{n}}),
\]
for a finite group $G$, when $\pi\in S_{n}$, and the number of alternating
cycles in the cycle graph $\mathcal{G}(\pi)$ (For the definition of cycle graph see \cite{Hultman1999}, \cite{DL}). It
was shown in \cite{CGLS} that this probability for a generic group $G$
depends only on the number of the cycles in the cycle graph $\mathcal{G}(\pi)$ and that
the spectrum of probabilities, as $\pi$ runs over all the permutations in
$S_{n}$, is the decomposition of $n!$ to the Hultman numbers. In this paper,
we generalize the results of \cite{CGLS}, where we consider the distribution
of the probabilities
\[
{\Pr}_{\pi_{\epsilon}}\left(  G \right)  =\Pr(a_{1}a_{2}\cdots a_{n}%
=a_{\pi_{1}}^{\epsilon_{1}}a_{\pi_{2}}^{\epsilon_{2}}\cdots a_{\pi_{n}%
}^{\epsilon_{n}}),
\]
for a finite group $G$, where $\pi\in S_{n}$, and $\epsilon_{i}\in\{+1, -1\}$ for every $1\leq i\leq n$.
Notice, the paper \cite{CGLS} deals with the case where $\epsilon_{i}=+1$ for every $1\leq i\leq n$.

In this paper, we show that ${\Pr}_{\pi_{\epsilon}}\left(  G\right)  $ for a
generic group $G$ depends only on the number of the cycles
 in the breakpoint graph $\Gr(\pi_{\epsilon})$  (which is a generalization of the cycle graph $\mathcal{G}(\pi)$) for a
signed permutation $\pi_{\epsilon}$ \cite{APA}, (where $\pi_{\epsilon}(i)=\epsilon
_{i}\cdot\pi(i)$ for every $\pi\in S_{n}$ and every $1\leq i\leq n$), and
that the spectrum of probabilities, as $\pi$ runs over all the permutations in
$S_{n}$, and $\epsilon$ runs over all the $n$-tuples of $\{+1,-1\}^{n}$, is
the decomposition of $2^{n}\cdot n!$ to signed Hultman numbers. Similarly to
the permutation group $S_{n}$, the signed permutation group is a Coxeter group
as well, which is denoted by $B_{n}$. The idea to generalize theorems
concerning permutation groups to signed permutation groups has a rich history.
For instance, there is a theorem about Coxeter covers of symmetric groups
\cite{rtv}, where characterizing special quotients of Coxeter groups are defined
by a line Coxeter graph as extensions of a symmetric group $S_{n}$. The theorem
was generalized into a theorem about Coxeter covers of classical Coxeter groups
\cite{AST}, where we consider characterization of special quotients of
particular Coxeter groups as extensions of a classical Coxeter group $B_{n}$
or $D_{n}$ (a subgroup of $B_{n}$, which can be considered as a specific
signed permutation group. For more details see \cite{BB}), by using
signed-graphs as a generalization of the dual line Coxeter graphs used in
\cite{rtv}. By the theorem of MacMahon \cite{Mac}, the major-index of the
permutations in $S_{n}$ is equi-distributed with the Coxeter length. Roichman
and Adin \cite{AR} proposed a generalization of the theorem of MacMahon, where
they defined the flag-major-index for signed permutations, which is a
generalization of the major-index for permutations, and they proved that the
flag-major-index is equi-distributed with the Coxeter length. A few years
later, Shwartz, Adin and Roichman \cite{SAR} generalized further the
flag-major-index for $D_{n}$ (even-signed permutations \cite{BB}), where they
proved that the flag-major-index is equi-distributed with the Coxeter length
of $D_{n}$ as well.

For Hultman numbers, signed Hultman numbers, and the related definitions and
notations we refer to \cite{DL}, \cite{APA}. Recall that the Hultman number
$S_{H(n,k)}$ counts the number of permutations $\pi$ in $S_{n}$ whose cycle
graph $\mathcal{G}(\pi)$ decomposes into $k$ alternating cycles. For a permutation $\pi$
in $S_{n}$, let $H(\pi)$ be the number of
alternating cycles in $\mathcal{G}(\pi)$.

The paper is organized as follows. In Section \ref{pre}, we give basic
definitions about groups, permutations, signed permutations, generalized
commuting probabilities, and signed generalized commuting probabilities, which
we use in the paper. In Section \ref{hult}, we recall the basic definitions
about breakpoint graphs, and the related alternating cycles for the
signed permutation groups $B_{n}$ as defined in \cite{APA}. In Section
\ref{exch-cyc}, we define $\left(  [x]--[y]\right)  $-exchange and $\left([x]--[y]\right)$-cyclic
operations on the breakpoint graph $\Gr(\pi)$ for
$\pi\in B_{n}$, as a generalization of the definitions of $(x--y)$-exchange
and $(x--y)$-cyclic operations on the cycle graph $\mathcal{G}(\pi)$ for $\pi\in S_{n}$, as defined in \cite{CGLS}. We also define a new operation, which has not
been defined or used yet in \cite{CGLS}, namely $\left([|x|]--[|x|+1]\right)$-sign-change operation. We show some important
properties of the mentioned operations on $\Gr(\pi)$, which we need in order describe
the connections between the number of cycles in the breakpoint graph $\Gr(\pi)$ for $\pi\in B_{n}$ and the
signed generalized commuting probability, which is induced by a
signed permutation $\pi\in B_{n}$. In Section \ref{main}, we prove the main
theorem of the paper, which draws a connection between the number of the
alternating cycle of the breakpoint graph $\Gr(\pi)$ for an element $\pi$ in
the signed permutation group $B_{n}$, and the signed generalized commuting
probability, which is induced by a signed permutation $\pi$. We show that the
main theorem is a generalization of the theorem about the connection between
the number of the alternating cycles of the cycle graph $\mathcal{G}(\pi)$ for an
element $\pi$ in the symmetric group $S_{n}$, and the generalized commuting
probability, which is induced by a signed permutation $\pi$, as described in
\cite{CGLS}. Moreover, we will observe two special families of finite groups
$G$, where the signed generalized commuting probability has interesting
properties. In Section \ref{conc}, we give our conclusions about the results
of the paper, which we compare to the results of \cite{CGLS}, since this paper
is a generalization of it. Finally, we offer ideas for further generalizations
of the generalized commuting probability.

\section{Preliminaries}

\label{pre}

We start with some important definitions about finite groups, which we use
throughout the paper.

\begin{definition}\label{group-basic}
Let $G$ be a finite group.

\begin{itemize}
\item For every $g,h\in G$, denote by $g^{h}$ the element $h^{-1}gh$, which is
the conjugate of $g$ by $h$.

\item For every $g\in G$, denote by $\Omega_{g}(G)$ the conjugacy class of $g$
in $G$, i.e. the set of all the elements of $G$ the form $h^{-1}gh$, where
$h\in G$.

\item For every $g\in G$, denote by $C_{G}(g)$ the centralizer of the element
$g\in G$, i.e. the subgroup of $G$, consisting of all the elements $h\in G$
such that $gh=hg$.

\item For every $g,h\in G$, denote by $[g,h]$ the commutator $g^{-1}h^{-1}gh$
of $g$ and $h$.

\item Denote by $c(G)$ the number of conjugacy classes in $G$.

\item Denote by $\Omega(G,1),...,\Omega(G,c(G))$ the conjugacy classes of $G$,
ordered in some fixed order, where $\Omega(G,1)$ is the
conjugacy class of the element $1\in G$.

\item For every $1\leq j\leq c(G)$, denote by $\Omega(G, j^{-1})$ the
conjugacy class, which contains the inverses of the elements of $\Omega(G, j)$.

\item For every $1\leq j\leq c(G)$, denote by $\Omega(G, j^{2})$ the conjugacy
class, which contains the squares of the elements of $\Omega(G, j)$.

Notice, the notation $\Omega(G, j^{2})$
nothing to do with the square of the number $j$, it just says that we consider the conjugacy class, which elements are the squares of the elements of the elements in $\Omega(G, j)$ (e.g. $\Omega(G, 3^{2})$ is not the same to  $\Omega(G, 9)$. The meaning of $\Omega(G, 3^{2})$ is the conjugacy class whose elements are the squares of the elements in $\Omega(G, 3)$.);

\item For every $g\in G$, denote by $ic_{g}$ the integer such that $g\in
\Omega(G, ic_{g})$;

\item For a sequence $(g_{1},g_{2},\dots,g_{n})$ of $n$ elements of $G$,
denote by\newline$\Stab_{n}(g_{1},g_{2},\dots,g_{n})$ the set of all the
sequences $(a_{1},a_{2},\dots,a_{n})$ of $n$ elements of $G$ such that
\[
a_{1}^{-1}g_{1}a_{1}\cdot a_{2}^{-1}g_{2}a_{2}\cdots a_{n}^{-1}g_{n}%
a_{n}=g_{1}\cdot g_{2}\cdots g_{n}.
\]

\item Denote by $c_{i_{1},...,i_{n};j}(G)$ the nonnegative integer number of
different ways of breaking an element $y\in\Omega(G,j)$ into a product
$y=x_{1}x_{2}\cdots x_{n}$, so that each $x_{t}$ for $1\leq t\leq n$, belongs
to the class $\Omega(G,i_{t})$.
\end{itemize}
\end{definition}

Notice that $c_{i_{1},...,i_{n};j}(G)$ does not depend on the choice of the
particular element $y$ from the set $\Omega(G,j)$.

\begin{proposition}
\label{c-i-j} \label{c} The following hold
\begin{itemize}
\item $c_{i_{1},...,i_{n};j}(G)=c_{i_{\phi(1)},...,i_{\phi(n)};j}(G)$ for
every permutation $\phi\in S_{n}$

\item $c_{i_{1},...,i_{n};j^{-1}}(G)=c_{i_{1}^{-1},...,i_{n}^{-1};j}(G)$

\item $c_{i_{1},i_{2},...,i_{n}, k_{1},k_{2},...,k_{n};1}(G)=\sum
\limits_{j=1}^{c(G)}\left\vert \Omega(G, j)\right\vert \cdot c_{i_{1}%
,i_{2},...,i_{n};j}(G)\cdot c_{k_{1},k_{2},...,k_{n};j^{-1}}(G).$
\end{itemize}
\end{proposition}

We get the proof of Proposition \ref{c-i-j} by using the definition of
$c_{i_{1},...,i_{n};j}(G)$, and basic group theoretical arguments.\newline

Notice that $\Stab_{n}(g_{1},g_{2},...,g_{n})$ is a generalization of the
notion of the centralizer of an element, and that $\Stab_{1}(g)$ is just
$C_{G}(g)$. \\ The size $\left\vert \Stab_{n}(g_{1},g_{2},...,g_{n}%
)\right\vert $ of $\Stab_{n}(g_{1},g_{2},...,g_{n})$ depends only on the
conjugacy classes of $g_{1},...,g_{n}$, and not on the elements themselves.
\newline\newline Recall that for all $g\in G$,
\[
\left\vert \Omega_{g}(G)\right\vert \cdot\left\vert C_{G}(g)\right\vert =
\left\vert G\right\vert
\]

\begin{definition}
\label{real} Let $G$ be a finite group.

\begin{itemize}
\item An element $g\in G$ is considered to be `real`, if $g$ is conjugate to
its inverse, $g^{-1}$.

\item $G$ is called an ambivalent group, if every element of $G$ is `real`.

\item A conjugacy class $\Omega$ of $G$ is called a `real conjugacy class`, if
every element of $g\in\Omega$ is a `real element`.

\item Denote by $rc(G)$ the number of `real conjugacy classes` of a finite
group $G$.
\end{itemize}
\end{definition}

By Definition \ref{real} it can be easily concluded that a finite group $G$ is
an ambivalent group if and only if every conjugacy class of $G$ is a `real
conjugacy class`. Hence, we conclude the following proposition.

\begin{proposition}\label{ambivalent-c}
Let $c_{i_{1},...,i_{n};j}(G)$ be a nonnegative integer as defined in the last part of Definition \ref{group-basic}. Then $c_{i_{1},...,i_{n};j^{-1}}(G)=c_{i_{1},...,i_{n};j}(G)$
for every conjugacy class $\Omega(G, j)$, if and only if $G$ is an
ambivalent group.

\end{proposition}

\begin{proof}
If $G$ is an ambivalent group, then every $g\in G$ is conjugate to its inverse $g^{-1}$. Thus $c_{i_{1},...,i_{n};j^{-1}}(G)=c_{i_{1},...,i_{n};j}(G)$
for every conjugacy class $\Omega(G, j)$ by the definition of $c_{i_{1},...,i_{n};j}(G)$.
Now, we assume $c_{i_{1},...,i_{n};j^{-1}}(G)=c_{i_{1},...,i_{n};j}(G)$
for every conjugacy class $\Omega(G, j)$, then in particular, $1=c(j;j)=c(j^{-1};j)$ for every conjugacy class $\Omega(G, j)$, which implies $G$ is an ambivalent group.
\end{proof}

\begin{definition}
Let $\inv(G)$ be the number of involutions in a finite group $G$; i.e, the
number of elements $b\in G$ such that $b^{2}=1$.
\end{definition}

Now, we recall the definition of signed permutation group, which is the
Coxeter group $B_{n}$ (for details see \cite{BB}).

\begin{definition}
\label{bn} For every $n\in\mathbb{N}$,

\begin{itemize}
\item Let $S_{n}$ be the permutation group of the elements of the set $\{1, 2,
\dots, n\}$;

\item Let $B_{n}$ be the permutation group of the elements of the set
\newline$\{+1,+2,...,+n,-1,-2,...,-n\}$ such that every permutation $\pi$
satisfies \newline$\pi_{-i}=-\pi_{+i}$ for every $1\leq i\leq n$, where
$-\left(  -i\right)  $ is defined to be $+i$.
\end{itemize}
\end{definition}

\begin{remark}
Definition \ref{bn} implies that every $\pi\in B_{n}$ is uniquely determined
by \newline$\pi_{+1},\pi_{+2},...,\pi_{+n}$. Therefore, we denote $\pi$ as
\[
\langle\pi_{+1}\;\;\pi_{+2}\;\;\ldots\;\;\pi_{+n}\rangle.
\]

\end{remark}

\begin{definition}
For every element $i\in\{+1,+2,...,+n\}$, let it be $\sign\left(i\right)=(+)$,
and for every element $i\in\{-1,-2,...,-n\}$, let it be $\sign\left(i\right)=(-)$.
\end{definition}

\begin{definition}
Let $\pi$ be a sign-permutation in $B_{n}$. We define $|\pi|$ as the corresponding
permutation of $S_{n}$, which satisfies $|\pi|_{i}=|\pi_{i}|$.
\end{definition}

Now, we recall some definition concerning generalized commuting probability.

\begin{definition}
\cite{CGLS}\label{Pr-power} ${\Pr}^{m}(G)=\Pr(a_{1}a_{2}\cdots a_{m}=a_{m}a_{m-1}\cdots
a_{1}).$
\end{definition}

\begin{proposition}
\cite{CGLS}\label{hultman-pr} For every finite group $G$ and every $\pi\in
S_{n}$, the number of alternating cycles $H(\pi)$
in the cycle graph $\mathcal{G}(\pi)$  is connected to the probability ${\Pr}
_{\pi}(G)$ by the following equality:
\[
{\Pr}_{\pi}(G)={\Pr}^{n-H(\pi)+1}(G).
\]

\end{proposition}

Now, we define signed generalized commuting probability, ${\Pr}_{\pi}(G)$ for
$\pi\in B_{n}$, which generalizes the definition of the generalized
commuting probability ${\Pr}_{\pi}(G)$ for $\pi\in S_{n}$ from
\cite{CGLS}.

\begin{definition}
Let $G$ be a finite group, and $\pi$ be a signed permutation in $B_{n}$.
Then
\[
{\Pr}_{\pi}(G):=\Pr(a_{1}a_{2}\cdots a_{n}=a_{|\pi|_{1}}^{\epsilon_{1}(\pi
)}a_{|\pi|_{2}}^{\epsilon_{2}(\pi)}\cdots a_{|\pi|_{n}}^{\epsilon_{n}(\pi)}),
\]
where for every $1\leq i\leq n$, either $\epsilon_{i}(\pi)=1$ or $\epsilon
_{i}(\pi)=-1$ such that:

\begin{itemize}
\item In case $\sign\left(\pi_{i}\right)=(+)$, $\epsilon_{i}(\pi)=1$;

\item In case $\sign\left(\pi_{i}\right)=(-)$, $\epsilon_{i}(\pi)=-1$.
\end{itemize}
\end{definition}

\begin{definition}
\label{m-pi-minus} Let $m(\pi)$ be the number of indices $1\leq i\leq n$, such
that $\epsilon_{i}(\pi)=-1$.
\end{definition}

\begin{definition}
A signed permutation $\pi$ considered as a positive element in case
\newline$m(\pi)=0$, which means $\epsilon_{i}(\pi)=1$ for every $1\leq i\leq
n$, otherwise $\pi$ is considered as a non-positive element.
\end{definition}

Now, we generalize the definition of ${\Pr}^{k}(G)$ from Definition \ref{Pr-power} for negative $k$ as well.

\begin{definition}
For every group $G$, and every $k\in\mathbb{N}$, define $\Pr^{-k}(G)$ to be
\[
{\Pr}^{-k}(G)=\Pr(a_{1}a_{2}\cdots a_{n}=a_{1}^{-1}a_{2}^{-1}\dots a_{k}^{-1}a_{k+1}a_{k+2}\cdots a_{n}).
\]

\end{definition}

\begin{definition}
\label{ik} Let $I^{(-k)}\in B_{n}$ the following signed permutation:

\begin{itemize}
\item $I^{(-k)}_{j}=-j$ for every $j$ such that $|j|\leq k$;

\item $I^{(-k)}_{j}=j$ for every $j$ such that $|j|>k$.
\end{itemize}

i.e., $I^{(-k)}=\langle-1\;\;\ldots\;\;-k\;\;+(k+1)\;\;+n\rangle$.
\end{definition}

\begin{remark}
${\Pr}^{-k}(G)=\Pr_{\pi}(G)$ for $\pi=I^{(-k)}$.
\end{remark}

\begin{proposition}
\label{pr-1}
Let $G$ be a finite group, then the following holds:
\begin{itemize}
\item
${\Pr}^{-1}(G)=\frac{\inv(G)}{|G|}$.

\item
${\Pr}^{-2}(G)=\frac{rc(G)}{|G|}$.
\end{itemize}
\end{proposition}

\begin{proof}
Let $G$ be a finite group, then
\[
{\Pr}^{-1}(G)=\Pr(a=a^{-1})=\Pr(a^{2}=1)=\frac{\inv(G)}{|G|}.
\]

\[
{\Pr}^{-2}(G)=\Pr(ab=a^{-1}b^{-1})=
\]
\[
=\Pr(ab=(ba)^{-1})=\Pr(b^{-1}(ba)b=(ba)^{-1})=
\]
\[
=\Pr(b^{-1}ab=a^{-1})=\frac{rc(G)}{|G|}
\]

\end{proof}

\begin{proposition}
\label{pr-minus-formula}
\[
{\Pr}^{-2n}(G)=\Pr\left(a_{1}a_{2}\cdots a_{2n-1}a_{2n}=a_{1}^{-1}a_{2}^{-1}\cdots a_{2n-1}^{-1}a_{2n}^{-1}\right)=
\]
\[
=\sum\limits_{i_{1},i_{2},...,i_{n}=1}^{c(G)}\frac{c_{i_{1},i_{2},...,i_{n}%
,i_{1},i_{2},...,i_{n};1}(G)}{\left\vert \Omega(G, i_{1})\right\vert
\cdot\left\vert \Omega(G, i_{2})\right\vert \cdot\cdots\cdot\left\vert
\Omega(G, i_{n})\right\vert \cdot\left\vert G\right\vert ^{n}}
\]

and

\[
{\Pr}^{-(2n+1)}(G)=\Pr\left(a_{1}a_{2}\cdots a_{2n-1}a_{2n}a_{2n+1}=a_{1}%
^{-1}a_{2}^{-1}\cdots a_{2n-1}^{-1}a_{2n}^{-1}a_{2n+1}^{-1}\right)=
\]
\[
=\sum\limits_{i_{1},i_{2},...,i_{n},j=1}^{c(G)}\frac{\left\vert \Omega(G, j^{2})\right\vert \cdot c_{i_{1},i_{2},...,i_{n},i_{1},i_{2},...,i_{n};j^{2}}%
(G)}{\left\vert \Omega(G, i_{1})\right\vert \cdot\left\vert \Omega(G, i_{2})\right\vert \cdot
\cdots\cdot\left\vert \Omega(G, i_{n})\right\vert
\cdot\left\vert G\right\vert ^{n}}.
\]

\end{proposition}

\begin{remark}
Since the proof of Proposition \ref{pr-minus-formula} is based on very similar
arguments as the proof of the formula for $\Pr^{2k}(G)$ in Section 5 of
\cite{CGLS}, we leave the proof of the proposition for the reader.
\end{remark}

\begin{remark}
Since ${\Pr}^{2}(G)=\frac{c(G)}{|G|}$, Proposition \ref{pr-1} implies that
${\Pr}^{-2}(G)$ equals to ${\Pr}^{2}(G)$ if and only if every conjugacy class
of a finite group $G$ is a 'real conjugacy class', which holds if and only if
$G$ is a finite ambivalent group. Therefore, ${\Pr}^{k}(G)$ is not necessarily
equal to ${\Pr}^{-k}(G)$ in general.
\end{remark}

From Proposition \ref{pr-minus-formula}, we conclude the following corollary,
which classifies the cases, where ${\Pr}^{k}(G)={\Pr}^{-k}(G)$.

\begin{corollary}
\label{ambivalent} For every integer $k\geq1$,
\[
{\Pr}^{2k}(G)={\Pr}^{-2k}(G)
\]
if and only if $G$ is an ambivalent group, which means that every element
$g\in G$ is conjugate to its inverse $g^{-1}\in G$.
\end{corollary}

\begin{proof}
By \cite{CGLS},
\[
{\Pr}^{2n}(G)=\Pr(a_{1}a_{2}\cdots a_{2n-1}a_{2n}=a_{2n}a_{2n-1}\cdots
a_{2}a_{1})=
\]
\[
=\sum\limits_{i_{1},i_{2},...,i_{n},j=1}^{c(G)}\frac{\left\vert
\Omega(G, j)\right\vert \cdot c_{i_{1},i_{2},...,i_{n};j}^{2}(G)}{\left\vert
\Omega(G, i_{1})\right\vert \cdot\left\vert \Omega(G, i_{2})\right\vert
\cdot\cdots\cdot\left\vert \Omega(G, i_{n})\right\vert \cdot\left\vert
G\right\vert ^{n}}
\]
and by Propositions \ref{c}, \ref{pr-minus-formula}%

\[
{\Pr}^{-2n}(G)=\Pr\left(a_{1}a_{2}\cdots a_{2n-1}a_{2n}=a_{1}^{-1}a_{2}%
^{-1}\cdots a_{2n-1}^{-1}a_{2n}^{-1}\right)=
\]
\[
=\sum\limits_{i_{1},i_{2},...,i_{n},j=1}^{c(G)}\frac{\left\vert
\Omega(G, j)\right\vert \cdot c_{i_{1},i_{2},...,i_{n};j}(G)\cdot c_{i_{1}%
,i_{2},...,i_{n};j^{-1}}(G)}{\left\vert \Omega(G, i_{1})\right\vert
\cdot\left\vert \Omega(G, i_{2})\right\vert \cdot\cdots\cdot\left\vert
\Omega(G, i_{n})\right\vert \cdot\left\vert G\right\vert ^{n}}=
\]%
\[
=\sum\limits_{i_{1},i_{2},...,i_{n},j=1}^{c(G)}\frac{\left\vert
\Omega(G, j)\right\vert \cdot c_{i_{1},i_{2},...,i_{n};j}(G)\cdot c_{i_{1}^{-1}%
,i_{2}^{-1},...,i_{n}^{-1};j}(G)}{\left\vert \Omega(G, i_{1})\right\vert
\cdot\left\vert \Omega(G, i_{2})\right\vert \cdot\cdots\cdot\left\vert
\Omega(G, i_{n})\right\vert \cdot\left\vert G\right\vert ^{n}}.
\]

Therefore, if $G$ is an ambivalent group, then by Proposition \ref{ambivalent-c}, $c_{i_{1},i_{2},...,i_{n};j}(G)=c_{i_{1},i_{2},...,i_{n}%
;j^{-1}}(G)$ for every conjugacy class $\Omega(G, j)$. Hence, ${\Pr}^{2k}(G)={\Pr}^{-2k}(G)$
for every positive integer $k$.
Now, we assume ${\Pr}^{-2k}(G)={\Pr}^{2k}(G)$ for every positive integer $k$.
Then ${\Pr}^{2}(G)={\Pr}^{-2}(G)$ in particular. Since by Proposition \ref{pr-1}, ${\Pr}^{-2}(G)=\frac{rc(G)}{|G|}$ and
by \cite{G, CGLS}, ${\Pr}^{2}(G)=\frac{k(G)}{|G|}$, we have $k(G)=rc(G)$, which implies that the
group $G$ is ambivalent.
\end{proof}

\begin{lemma}
\label{squares_product}
Let $G$ be a finite group, then the following holds:
\[
{\Pr}^{-k}(G)=\Pr\left(a_{1}^{2}a_{2}^{2}\cdots a_{k}^{2}=1\right).
\]

\end{lemma}

\begin{proof}
We have

\[
\Pr(a_{1}a_{2}\dots a_{k}=a_{1}^{-1}a_{2}^{-1}\cdots a_{k}^{-1}),
\]
which is equivalent to

\[
\Pr\left(a_{k}^{2} ~ (a_{k}^{-1}a_{k-1}^{2}a_{k}) ~ (a_{k}^{-1}a_{k-1}^{-1}%
a_{k-2}^{2}a_{k-1}a_{k}) ~ \cdots~ (a_{k}^{-1}a_{k-1}^{-1}\cdots a_{2}%
^{-1}a_{1}^{2}a_{2}\cdots a_{k-1}a_{k})=1\right).
\]
Then by substituting $b_{i}=a_{k-i+1}^{\prod_{j=k-i+2}%
^{k}a_{j}}$ for every $1\leq i\leq k$, we have
\[
\Pr(a_{1}a_{2}\dots a_{k}=a_{1}^{-1}a_{2}^{-1}\cdots a_{k}^{-1})=\Pr(b_{1}^{2}b_{2}^{2}\cdots b_{k}^{2}=1).
\]

\end{proof}

From Lemma \ref{squares_product}, we conclude the following corollaries:

\begin{corollary}
\label{odd-order} A finite group $G$ has an odd order, if and only if
\[
{\Pr}^{-k}(G)=\frac{1}{|G|}%
\]
for every $k\in\mathbb{N}$.
\end{corollary}

\begin{proof}
Let $G$ be a finite group. Then for every $g\in G$ there exists only one
$g^{\prime}\in G$ such that ${g^{\prime}}^{2}=g$, if and only if the order of
$G$ is odd. By Lemma \ref{squares_product}, \newline${\Pr}^{-k}(G)=\Pr\left(a_{1}^{2}a_{2}^{2}\cdots a_{k}^{2}=1\right)$. Therefore, in case of group $G$ the
order of which an odd integer, ${\Pr}^{-k}(G)=\Pr(a^{\prime}_{1}a^{\prime}_{2}\cdots
a^{\prime}_{k}=1)$. Hence, obviously, ${\Pr}^{-k}(G)=\frac{1}{|G|}$ in case of
odd order $G$. In case of $G$ the order of which is an even integer, by Sylow
Theorem, the group $G$ contains at least one non-trivial involution (i.e., an
element of order $2$ in $G$). Thus, by Proposition \ref{pr-1},
\[
{\Pr}^{-1}(G)=\frac{\inv(G)}{|G|}>\frac{1}{|G|}.
\]

\end{proof}

\begin{corollary}
\label{abelian} If $G$ is a finite abelian group, then
\[
{\Pr}^{-k}(G)={\Pr}^{-1}(G)=\frac{\inv(G)}{|G|},
\]
for every $k\in\mathbb{N}$.
\end{corollary}

\begin{proof}
By Lemma \ref{squares_product}, ~${\Pr}^{-k}(G)=\Pr\left(a_{1}^{2}a_{2}^{2}\cdots a_{k}^{2}=1\right)$. Since $G$ is an abelian group, $a_{1}^{2}a_{2}^{2}\cdots a_{k}^{2}=\left(a_{1}a_{2}\cdots a_{k}\right)^{2}$. Thus,
\[
{\Pr}^{-k}(G)=\Pr\left(\left(a_{1}a_{2}\cdots a_{k}\right)^{2}=1\right)=\Pr(a^{2}=1)={\Pr}^{-1}(G)=\frac{\inv(G)}{|G|}.
\]

\end{proof}

\begin{corollary}
\label{direct-sum} Let $G=\bigoplus_{i=1}^{m}G_{i}$ (i.e., $G$ is a direct sum
of $m$ groups $G_{i}$ for $1\leq i\leq m$), then
\[
{\Pr}^{-k}(G)=\prod_{i=1}^{m}{\Pr}^{-k}(G_{i}),
\]
for every $k\in\mathbb{N}$.
\end{corollary}

\begin{proof}
By Lemma \ref{squares_product}, ~${\Pr}^{-k}(G)=\Pr\left(a_{1}^{2}a_{2}^{2}\cdots a_{k}^{2}=1\right)$. Since $G=\bigoplus_{i=1}^{m}G_{i}$, it is satisfied that
$a_{j}=\prod_{i=1}^{m}b_{j,i}$ such that $b_{j,i}\in G_{i}$ for every $1\leq
j\leq k$. Thus we conclude:
\[
{\Pr}^{-k}(G)=\Pr\left(a_{1}^{2}a_{2}^{2}\cdots a_{k}^{2}=1\right)=\Pr(\prod_{i=1}^{m}\prod_{j=1}^{k}b_{j,i}^{2}=1)=\prod_{i=1}^{m}{\Pr}^{-k}(G_{i}).
\]

\end{proof}

\section{Breakpoint graphs}

\label{hult}

Now, we recall the definition of the breakpoint graph for the
signed permutations of $B_{n}$, as defined in \cite{APA}. The breakpoint graph
is a generalization of the cycle graph for permutations
of $S_{n}$, which is defined in \cite{Hultman1999}, \cite{DL}.

\begin{definition}
\label{hultman} Let $\pi\in B_{n}$. Consider the set $H_{n}$ of $2n+2$
vertices named by
\[
H_{n}=\{+0, +1, ..., +n, -0, -1, ..., -n\}.
\]
\begin{remark}
The elements $+0$ and $-0$ are two distinct elements of $H_{n}$ (which are denoted $0^{h}$ and $0^{t}$ in \cite{APA}).
\end{remark}

Defining two types of edges.

\begin{itemize}
\item There are gray-edges connecting
\[
[i]\leftrightsquigarrow[-(i+1)]
\]
for every $i\in H_{n}$;

\item There are black-edges connecting
\[
[\pi_{i}]\longleftrightarrow[\pi_{-(i+1)}]
\]
for every $i\in H_{n}$,
\end{itemize}

such that in the specific case, where $i$ or $i+1$ equals to $(+0)$ or $(-0)$,
$i+1$ is defined as follows:

\begin{itemize}
\item $(-1)+1$ is considered to be $(-0)$;

\item $(-0)+1$ is considered to be $(-n)$;

\item $(+n)+1$ is considered to be $(+0)$.
\end{itemize}

Then considering the breakpoint graph $\Gr(\pi)$, the edges of which are two-colored
containing the black and the gray non-oriented edges.
\end{definition}

\begin{remark}
Notice, both the gray and the black edges are non-oriented in the breakpoint
graph $\Gr(\pi)$ for $\pi\in B_{n}$, in contrast to the edges in the cycle
graph $\mathcal{G}(\pi)$ for $\pi\in S_{n}$ as described in \cite{BP}, \cite{DL}, and
\cite{CGLS}.
\end{remark}

\begin{definition}
\label{number-alternating-cycle} Let $s(\pi)$ be the number of the
cycles of $\Gr(\pi)$, where in every cycle black edge is followed by gray edge, and
gray edge is followed by black edge.
\end{definition}

\begin{example}
Let $\pi=\langle3\;\;-1\;\;2\;\;4\rangle$. Then $\pi$ contains two
cycles, in the following way:
\[
[+0]\leftrightsquigarrow[-1]\longleftrightarrow[-2]\leftrightsquigarrow
[+1]\longleftrightarrow[+3]\leftrightsquigarrow[-4]\longleftrightarrow
[+2]\leftrightsquigarrow[-3]\longleftrightarrow[+0]
\]

\begin{center}
and
\end{center}

\[
[+4]\leftrightsquigarrow[-0]\longleftrightarrow[+4].
\]

Therefore, $s(\pi)=2$.
\end{example}

\begin{remark}
Our notation of the vertices of $\Gr(\pi)$ is slightly different from the
notation in \cite{APA}, where for every integer $0\leq i\leq n$,

\begin{itemize}
\item we denote by ~$+i$ ~the vertex denoted by ~$i^{h}$ ~in \cite{APA};

\item we denote by ~$-i$ ~the vertex denoted by ~$i^{t}$ ~in \cite{APA}.
\end{itemize}
\end{remark}

\begin{proposition}
\label{i-k} Let $I^{(k)}$ be a signed permutation of $B_{n}$ as defined in
Definition \ref{ik}. Then $s(I^{(-k)})=n-k+1$ (i.e., $\Gr(I^{(-k)})$ contains
$n-k+1$ alternating cycles).
\end{proposition}

\begin{proof}
By the definition of $I^{(k)}$, and the definition of the alternating cycle,
we get $n-k$ alternating cycles, where each one contains two vertices $+j$ and
$-j-1$ for every $k+1\leq j\leq n$, and one more alternating cycle, which
contains the remaining $2k+1$ vertices, which are $j$ such that $+0\leq j\leq+k$,
and $-k\leq j\leq-1$.
\end{proof}

\begin{definition}
\label{bullet-bullet} Denote by $\pi^{\ast}$ and by $\pi^{\circ}$ the following elements of $S_{H_{n}}$
\begin{itemize}
\item
$\pi^{\ast}=(\pi_{+n}, ~\pi_{+(n-1)}, ~\pi_{+(n-2)}, ..., ~\pi_{+0})\cdot(\pi_{-0}, ~\pi_{-1}, ~\pi_{-2}, ..., ~\pi_{-n})$
;

\item
$\pi^{\circ}=\pi^{\ast}\cdot(+0, ~+1, ..., ~+n)\cdot(-n,...,~-1, ~-0)$.
\end{itemize}

\end{definition}

\begin{proposition}
\label{bullet-bullet-m} Let $\pi\in B_{n}$, then there are $m(\pi)$ elements
$k$ such that $\sign\left(  k\right)  =(-)$ in the same cycle to ~$+0$ ~of
$\pi^{\ast}$.
\end{proposition}

\begin{proof}
The cycle of $\pi^{\ast}$, where ~$+0$ ~is located has the form
\newline$(\pi_{+n},~\pi_{+(n-1)},~\pi_{+(n-2)},...,~\pi_{+0})$. Hence, by
Corollary \ref{m-pi-minus}, the number of elements with sign $(-)$ in that
cycle equals to $m(\pi)$.
\end{proof}

\begin{corollary}
\label{perm-alternating} Let $\pi\in B_{n}$. Then we have the following
connections between $\pi^{\circ}$ and the breakpoint graph $\Gr(\pi)$:
\begin{itemize}
\item For $x,y\in H_{n}$, ~$\pi_{x}^{\circ}=y$ if and only if the cycle graph
$\Gr(\pi)$ contains \newline$[x]\leftrightsquigarrow\lbrack
-(x+1)]\longleftrightarrow\lbrack y]$;

\item The number of cycles of $\pi^{\circ}$ as a permutation of the $2n+2$
elements of $H_{n}$ equals to $2s(\pi)$.
\end{itemize}
\end{corollary}

\section{Exchange, cyclic, and sign-change operations}

\label{exch-cyc}

First, we recall and generalize the definitions of $\left([x]--[y]\right)$-exchange, and $\left([x]--[y]\right)$-cyclic operations on $\Gr(\pi)$,
which were defined in \cite{CGLS}. In addition, we define
\newline$\left([|x|]--[|x|+1]\right)$-sign-exchange operation. We use the operations for
the proof of the theorems about the connections between the signed generalized
commuting probability $\Pr_{\pi}(G)$ and the number of alternating cycles
$s(\pi)$ in $\Gr(\pi)$, similarly to its use in \cite{CGLS}.

\begin{definition}
\label{def.exchange} Let $\pi\in B_{n}$ a signed permutation. Let $x,y,w,z\in
H_{n}$ such that $y\neq x$, ~$y\neq x+1$, and
\[
[-z]\longleftrightarrow[x]\leftrightsquigarrow[-(x+1)]\longleftrightarrow[y],
~~~~ [-y]\longleftrightarrow[w]
\]
is satisfied in $\Gr(\pi)$. Then by $\left(  [x]--[y]\right)  $-exchange
operation on $\pi$ we obtain $\theta$ such that in $\Gr(\theta)$ the following
holds:
\[
[-z]\longleftrightarrow[y], ~~~~ [-y]\longleftrightarrow
[x]\leftrightsquigarrow[-(x+1)]\longleftrightarrow[w]
\]
and all the other arrows of $\Gr(\theta)$ are the same as in $\Gr(\pi)$.
\end{definition}

\begin{proposition}
\label{obs.exchange} Let $\pi\in B_{n}$, ~$x,y,z,w\in H_{n}$ satisfy
the conditions of Definition \ref{def.exchange}, and $\theta\in B_{n}$ is
obtainable from $\pi$ by an $\left(  [x]--[y]\right)  $-exchange operation.
Then the following properties hold:

\begin{itemize}
\item If $x=w$ or $y=z$ then $\theta=\pi$ (i.e., the $\left(  [x]--[y]\right)
$-exchange operation does not do anything to $\pi$);

\item If $m(\pi)=0$ (i.e., $\pi(i)>0$ for every $1\leq i\leq n$), then
$m(\theta)=0$ as well such that $|\theta|$ is obtained from $|\pi|$ by
$\left(  x--y\right)  $-exchange operation as defined in \cite{CGLS}, by the
same $x$ and $y$;

\item $\theta^{\ast}=(x, ~y, ~w)(-y, ~-(x+1), ~-z)\cdot\pi^{\ast}$
(i.e., obtaining $\theta$ by an $\left(  [x]--[y]\right)  $-exchange operation
on $\pi$ changes just the location of $y$ from being between $x+1$ and $w$ in
a cycle of $\pi^{\ast}$, to being located between $z$ and $x$ in a cycle of
$\theta^{\ast}$, and the location of $-y$ from being between $-w$ and
$-(x+1)$ in a cycle of $\pi^{\ast}$ to being located between $-x$ and $-z$ in
a cycle of $\theta^{\ast}$);

\item If ~$\sign\left(  \pi^{-1}_{x}\right)  \neq \sign\left(  \pi^{-1}%
_{x+1}\right)  $, ~then either ~~$m\left(  \theta\right)  =m\left(  \pi\right)
-1$ ~or ~$m\left(  \theta\right)  =m\left(  \pi\right)  +1$;

\item If ~$\sign\left(  \pi^{-1}_{x}\right)  =\sign\left(  \pi^{-1}%
_{x+1}\right)  $, ~then ~~$m\left(  \theta\right)  =m\left(  \pi\right)  $;

\item $s(\pi)=s(\theta)$;

\item ${\Pr}_{\pi}(G)={\Pr}_{\theta}(G)$;

\item Obtaining $\theta$ by performing a $\left(  [x]--[y]\right)  $-exchange
operation on $\pi$ if and only if obtaining $\pi$ by performing  a $\left(
[-(x+1)]--[-y]\right)  $-exchange operation on $\theta$.
\end{itemize}
\end{proposition}

\begin{proof}
The proof of most statements of the proposition is a direct consequence of
the definition. The proof of the part ${\Pr}_{\pi}(G)={\Pr}_{\theta}(G)$, in
case $\theta$ is obtainable from $\pi$ by a $\left(  [x]--[y]\right)
$-exchange operation, can be proved by the same argument as in the case of
$S_{n}$, which has been proved in Section 4.1. of \cite{CGLS}.
\end{proof}

\begin{example}
Consider the following signed permutation $\pi\in B_{n}$

\begin{itemize}
\item $\pi=\langle+6\;\;+2\;\;-3\;\;+4\;\;-5\;\;+1\rangle$. Then
\[
\pi^{\ast}=(+0, +1, -5, +4, -3, +2, +6)\cdot(-6, -2, +3, -4, +5, -1, -0).
\]
Thus,
\begin{align*}
\pi^{\circ}  &  =\pi^{\ast}\cdot(+0, +1, +2, +3, +4, +5, +6)\cdot(-6, -5,
-4, -3, -2, -1, -0)\\
&  =(+0, -5, +5)\cdot(-6, +4, -1)\cdot(+1, +6)\cdot(-0, -2)\cdot(+2,
-4)\cdot(+3, -3)
\end{align*}
and
\begin{align*}
\Gr(\pi): ~~  &  [+0]\leftrightsquigarrow[-1]\longleftrightarrow
[-5]\leftrightsquigarrow[+4]\longleftrightarrow[+5]\leftrightsquigarrow
[-6]\longleftrightarrow[+0]\\
&  [+1]\leftrightsquigarrow[-2]\longleftrightarrow[+6]\leftrightsquigarrow
[-0]\longleftrightarrow[+1]\\
&  [+2]\leftrightsquigarrow[-3]\longleftrightarrow[-4]\leftrightsquigarrow
[+3]\longleftrightarrow[+2].
\end{align*}
By performing a $\left(  [+2]--[-4]\right)  $-exchange operation on $\pi$
\newline(we have: $x=(+2), ~y=(-4), ~w=(+5), ~z=(-3)$) we obtain $\theta$ such
that
\begin{align*}
\theta^{\ast}  &  =(+2, -4, +5)\cdot(+4, -3, +3)\cdot\pi^{\ast}\\
&  =(+0, +1, -5, -3, -4, +2, +6)\cdot(-6, -2, +4, +3, +5, -1, -0).
\end{align*}
Thus,
\begin{align*}
\theta^{\circ}  &  =\theta^{\ast}\cdot(+0, +1, +2, +3, +4, +5, +6)\cdot(-6,
-5, -4, -3, -2, -1, -0)\\
&  =(+0, -5, +2, +5)\cdot(-6, -3, +4, -1)\cdot(+1, +6)\cdot(-0, -2)\cdot
(-4)\cdot(+3)
\end{align*}
and
\begin{align*}
\Gr(\theta): ~~  &  [+0]\leftrightsquigarrow[-1]\longleftrightarrow
[-5]\leftrightsquigarrow[+4]\longleftrightarrow[+2]\leftrightsquigarrow
[-3]\longleftrightarrow[+5]\leftrightsquigarrow[-6]\longleftrightarrow[+0]\\
&  [+1]\leftrightsquigarrow[-2]\longleftrightarrow[+6]\leftrightsquigarrow
[-0]\longleftrightarrow[+1]\\
&  [-4]\leftrightsquigarrow[+3]\longleftrightarrow[-4].
\end{align*}
Hence, ~$\theta=\langle+6\;\;+2\;\;-4\;\;-3\;\;-5\;\;+1\rangle$. By performing
a $\left(  [-3]--[+4]\right)  $-exchange operation on $\theta$ we obtain
$\pi$;

\item $\pi=\langle+5\;\;+3\;\;+2\;\;+6\;\;-4\;\;-1\rangle$. Then
\[
\pi^{\ast}=(+0, -1, -4, +6, +2, +3, +5)\cdot(-5, -3, -2, -6, +4, +1, -0).
\]
Thus,
\begin{align*}
\pi^{\circ}  &  =\pi^{\ast}\cdot(+0, +1, +2, +3, +4, +5, +6)\cdot(-6, -5,
-4, -3, -2, -1, -0)\\
&  =(+0, -0, +4)\cdot(-5, +6, -1)\cdot(+1, +3)\cdot(-4, -2)\cdot(+2,
+5)\cdot(-6, -3)
\end{align*}
and
\begin{align*}
\Gr(\pi): ~~  &  [+0]\leftrightsquigarrow[-1]\longleftrightarrow
[-0]\leftrightsquigarrow[+6]\longleftrightarrow[+4]\leftrightsquigarrow
[-5]\longleftrightarrow[+0]\\
&  [+1]\leftrightsquigarrow[-2]\longleftrightarrow[+3]\leftrightsquigarrow
[-4]\longleftrightarrow[+1]\\
&  [+2]\leftrightsquigarrow[-3]\longleftrightarrow[+5]\leftrightsquigarrow
[-6]\longleftrightarrow[+2].
\end{align*}
By performing a $\left(  [+3]--[+1]\right)  $-exchange operation on $\pi$
\newline(we have: $x=(+3), ~y=(+1), ~w=(-0), ~z=(+2)$) we obtain $\theta$ such
that
\begin{align*}
\theta^{\ast}  &  =(+3, +1, -0)\cdot(-1, -4, -2)\cdot\pi^{\ast}\\
&  =(+0, -4, +6, +2, +1, +3, +5)\cdot(-5, -3, -1, -2, -6, +4, -0).
\end{align*}
Thus,
\begin{align*}
\theta^{\circ}  &  =\theta^{\ast}\cdot(+0, +1, +2, +3, +4, +5, +6)\cdot(-6,
-5, -4, -3, -2, -1, -0)\\
&  = (+0, +3, -0, +4)\cdot(-5, +6, -4, -1)\cdot(+1)\cdot(-2)\cdot(+2,
+5)\cdot(-6, -3)
\end{align*}
and
\begin{align*}
\Gr(\theta): ~~  &  [+0]\leftrightsquigarrow[-1]\longleftrightarrow
[+3]\leftrightsquigarrow[-4]\longleftrightarrow[-0]\leftrightsquigarrow
[+6]\longleftrightarrow[+4]\leftrightsquigarrow[-5]\longleftrightarrow[+0]\\
&  [+1]\leftrightsquigarrow[-2]\longleftrightarrow[+1]\\
&  [+2]\leftrightsquigarrow[-3]\longleftrightarrow[+5]\leftrightsquigarrow
[-6]\longleftrightarrow[+2].
\end{align*}
Hence, ~$\theta=\langle+5\;\;+3\;\;+1\;\;+2\;\;+6\;\;-4\rangle$, then by
performing a $\left(  [-4]--[-1]\right)  $-exchange operation on $\theta$, we
obtain $\pi$.
\end{itemize}
\end{example}

\begin{definition}
\label{cyclic-cdot} Let $\pi\in B_{n}$ such that $\pi^{\ast}\in S(2+2n)$
contains one of the following:

\begin{itemize}
\item $x+1\rightarrow y\rightarrow x$, ~(If $x=+n$, then it becomes $+0\rightarrow
y\rightarrow+n$), and assume $\sign\left(x\right)=(+)$;

\item $-x-1\rightarrow y\rightarrow x$, ~(If $x=+n$, then it becomes $-0\rightarrow
y\rightarrow+n$), and assume $\sign\left(x\right)=(+)$;

\item $-x+1\rightarrow y\rightarrow x$, ~(If $x=-n$, then it becomes $+0\rightarrow
y\rightarrow-n$), and assume $\sign\left(x\right)=(-)$.
\end{itemize}

Then we define $\left(  [x]--[y]\right)  $-cyclic operation as follows (in case
$\pi^{\ast}$ contains ~$x+1\rightarrow y\rightarrow x$ and $\sign\left(
x\right)  =\sign\left(  y\right)=(+)$, the definition is the same as in
\cite{CGLS}):

\begin{itemize}
\item If $\sign\left(  y\right)  =(+)$, ~$x+1\rightarrow y$ in $\pi^{\ast}$,
and $y>x+1$, then in $\pi^{\ast}$ we replace $x+1$ with $y-1$ and each $t$,
where $t=x+2,...,y-1$, we replace with $t-1$. We also replace $-x-1$ with
$-y-1$ and each $t$, where $t=-x-2,...,-y+1$, we replace with $t+1$;

\item If $\sign\left(  y\right)  =(+)$, ~$x+1\rightarrow y$ in $\pi^{\ast}$,
and $y<x$, then in $\pi^{\ast}$ we replace $x$ with $y+1$ and each $t$ for
$t=y+1,...,x-1$, we replace with $t+1$. We also replace $-x$ with $-y-1$ and
each $t$ for $t=-y-1,...,-x+1$, we replace with $t-1$;

\item If $\sign\left(  y\right)  =(-)$, ~$x+1\rightarrow y$ in $\pi^{\ast}$,
and $|y|>x+1$, then in $\pi^{\ast}$ we replace $x+1$ with $y$ and each $t$,
where $t=x+2,...,|y|$, we replace with $t-1$. We also replace $-x-1$ with $-y$
and each $t$, where $t=-x-2,...,-|y|$, we replace with $t+1$;

\item If $\sign\left(  y\right)  =(-)$, ~$x+1\rightarrow y$ in $\pi^{\ast}$,
and $|y|<x$, then in $\pi^{\ast}$ we replace $x$ with $y$ and each $t$ for
$t=|y|+1,...,x-1$, we replace with $t+1$. We also replace $-x$ with $-y$ and
each $t$ for $t=-|y|-1,...,-x+1$, we replace with $t-1$;

\item If $\sign\left(  y\right)  =(+)$, ~$-x-1\rightarrow y$ in $\pi^{\ast}%
$, and $y>x+1$, then in $\pi^{\ast}$ we replace $x+1$ with $y-1$ and each
$t$, where $t=x+2,...,y-1$, we replace with $t-1$. We also replace $-x-1$ with
$-y+1$ and each $t$, where $t=-x-2,...,-y+1$, we replace with $t+1$;

\item If $\sign\left(  y\right)  =(+)$, ~$-x-1\rightarrow y$ in $\pi^{\ast}%
$, and $y<x$, then in $\pi^{\ast}$ we replace $x+1$ with $y$ and each $t$,
for $t=y,...,x$, we replace with $t+1$. We also replace $-x-1$ with $-y$ and
each $t$ for $t=-y,...,-x$, we replace with $t-1$;

\item If $\sign\left(  y\right)  =(-)$, ~$-x-1\rightarrow y$ in $\pi^{\ast}%
$, and $|y|>x+1$, then in $\pi^{\ast}$ we replace $x+1$ with $y$ and each
$t$, where $t=x+2,...,|y|$, we replace with $t-1$. We also replace $-x-1$ with
$-y$ and each $t$, where $t=-x-2,...,-|y|$, we replace with $t+1$;

\item If $\sign\left(  y\right)  =(-)$, ~$-x-1\rightarrow y$ in $\pi^{\ast}%
$, and $|y|<x$, then in $\pi^{\ast}$ we replace $x+1$ with $y-1$ and each
$t$ for $t=|y|+1,...,x$, we replace with $t+1$. We also replace $-x-1$ with
$-y+1$ and each $t$ for $t=-|y|-1,...,-x$, we replace with $t-1$;

\item If $\sign\left(  y\right)  =(+)$, ~$-x+1\rightarrow y$ in $\pi^{\ast}%
$, and $y>|x|+1$, then in $\pi^{\ast}$ we replace $|x|$ with $y$ and each
$t$, where $t=|x|+1,...,y$, we replace with $t-1$. We also replace $-|x|$ with
$-y$ and each $t$, where $t=-|x|-1,...,-y$, we replace with $t+1$;

\item If $\sign\left(  y\right)  =(+)$, ~$-x+1\rightarrow y$ in $\pi^{\ast}%
$, and $y<|x|$, then in $\pi^{\ast}$ we replace $|x|$ with $y+1$ and each
$t$ for $t=y+1,...,|x|-1$, we replace with $t+1$. We also replace $-|x|$ with
$-y-1$ and each $t$ for $t=-y-1,...,-|x|+1$, we replace with $t-1$;

\item If $\sign\left(  y\right)  =(-)$, ~$-x+1\rightarrow y$ in $\pi^{\ast}%
$, and $|y|>|x|+1$, then in $\pi^{\ast}$ we replace $|x|$ with $y+1$ and
each $t$, where $t=|x|+1,...,|y|-1$, we replace with $t-1$. We also replace
$-|x|$ with $-|y|-1$ and each $t$, where $t=-|x|-1,...,-|y|+1$, we replace
with $t+1$;

\item If $\sign\left(  y\right)  =(-)$, ~$-x+1\rightarrow y$ in $\pi^{\ast}%
$, and $|y|<|x|$, then in $\pi^{\ast}$ we replace $|x|$ with $y$ and each
$t$ for $t=|y|,...,|x|-1$, we replace with $t+1$. We also replace $-|x|$ with
$-y$ and each $t$ for $t=-|y|,...,-|x|+1$, we replace with $t-1$;
\end{itemize}
\end{definition}

\begin{observation}
Notice the following observations concerning $\sign\left(  x\right)  $:

\begin{itemize}
\item Assume $\sign\left(  x\right)  =(+)$ and $\pi^{\ast}$ contains
$x+1\rightarrow y\rightarrow x$. Then by Definition \ref{bullet-bullet},
$\pi^{\ast}$ contains $-x\rightarrow-y\rightarrow-x-1$ as well. Thus, the
case of $\pi^{\ast}$ contains $x^{\prime}+1\rightarrow y^{\prime
}\rightarrow x^{\prime}$, with $\sign\left(  x^{\prime}\right)  =(-)$ we
conclude by substituting $x^{\prime}=-x-1$;

\item Assume $\sign\left(  x\right)  =(+)$ and $\pi^{\ast}$ contains
$-x-1\rightarrow y\rightarrow x$. Then by Definition \ref{bullet-bullet},
$\pi^{\ast}$ contains $-x\rightarrow-y\rightarrow x+1$ as well. Thus, the
case of $\pi^{\ast}$ contains \newline$-x^{\prime}+1\rightarrow y^{\prime
}\rightarrow x^{\prime}$, with $\sign\left(  x^{\prime}\right)  =(+)$ we
conclude by substituting $x^{\prime}=x+1$;

\item Assume $\sign\left(  x\right)  =(-)$ and $\pi^{\ast}$ contains
$-x+1\rightarrow y\rightarrow x$. Then by Definition \ref{bullet-bullet},
$\pi^{\ast}$ contains $-x\rightarrow-y\rightarrow x-1$ as well. Thus, the
case of $\pi^{\ast}$ contains \newline$-x^{\prime}-1\rightarrow y^{\prime
}\rightarrow x^{\prime}$, with $\sign\left(  x^{\prime}\right)  =(-)$ we
conclude by substituting $x^{\prime}=x-1$.
\end{itemize}
\end{observation}

\begin{proposition}
\label{obs.cyclic} Let $\pi$ and $\theta$ be two signed permutations of
$B_{n}$ such that $\theta$ is obtainable from $\pi$ by a $\left(
[x]--[y]\right)  $-cyclic operation. Then the following properties hold:

\begin{itemize}
\item If $\sign\left(  y\right)  =(+)$, then ~$m(\theta)=m(\pi)$;

\item If $\sign\left(  y\right)  =(-)$, then either ~$m(\theta)=m(\pi)-1$ ~or
~$m(\theta)=m(\pi)+1$;

\item $s(\theta)=s(\pi)$;

\item ${\Pr}_{\pi}(G)={\Pr}_{\theta}(G)$.
\end{itemize}
\end{proposition}

\begin{proof}
The proof of the first three parts of the proposition comes directly from the
definition of $\left(  [x]--[y]\right)  $-cyclic operation as defined in
Definition \ref{cyclic-cdot}. The last part has been proved partially in
Section 4.2. of \cite{CGLS} (the case, where $\pi^{\ast}$ contains
$x+1\rightarrow y\rightarrow x$, and where $\sign\left(  y\right)  =(+)$), and
for the other cases, it can be proved by very similar arguments.
\end{proof}

\begin{example}
Consider the following signed permutation $\pi\in B_{n}$

\begin{itemize}
\item $\pi=\langle-2\;\;-6\;\;+3\;\;+1\;\;+5\;\;+4\rangle$. Then
\[
\pi^{\ast}=(+0, +4, +5, +1, +3, -6, -2)\cdot(+2, +6, -3, -1, -5, -4, -0).
\]
We can perform $\left(  [-2]--[-6]\right)  $-cyclic operation on $\pi$ and
obtain a new permutation $\theta$. We have
\[
\theta^{\ast}=(+0, +3, +4, +1, +2, -6, +5)\cdot(-5, +6, -2, -1, -4, -3,
-0).
\]
Hence, $\theta=\langle+5\;\;-6\;\;+2\;\;+1\;\;+4\;\;+3\rangle$;

\item $\pi=\langle-5\;\;+3\;\;+6\;\;+1\;\;+4\;\;+2\rangle$. Then
\[
\pi^{\ast}=(+0, +2, +4, +1, +6, +3, -5)\cdot(+5, -3, -6, -1, -4, -2, -0).
\]
We can perform $\left(  [-5]--[+3]\right)  $-cyclic operation on $\pi$ and
obtain a new permutation $\theta$. We have
\[
\theta^{\ast}=(+0, +2, +5, +1, +6, +3, -4)\cdot(+4, -3, -6, -1, -5, -2,
-0).
\]
Hence, $\theta=\langle-4\;\;+3\;\;+6\;\;+1\;\;+5\;\;+2\rangle$.

\item $\pi=\langle+2\;\;-6\;\;-3\;\;+1\;\;+5\;\;+4\rangle$. Then
\[
\pi^{\ast}=(+0, +4, +5, +1, -3, -6, +2)\cdot(-2, +6, +3, -1, -5, -4, -0).
\]
We can perform $\left(  [+2]--[-6]\right)  $-cyclic operation on $\pi$ and
obtain a new permutation $\theta$. We have
\[
\theta^{\ast}=(+0, +3, +4, +1, +6, -5, +2)\cdot(-2, +5, -6, -1, -4, -3,
-0).
\]
Hence, $\theta=\langle+2\;\;-5\;\;+6\;\;+1\;\;+4\;\;+3\rangle$;

\item $\pi=\langle+5\;\;+3\;\;-6\;\;+1\;\;+4\;\;+2\rangle$. Then
\[
\pi^{\ast}=(+0, +2, +4, +1, -6, +3, +5)\cdot(-5, -3, +6, -1, -4, -2, -0).
\]
We can perform $\left(  [+5]--[+3]\right)  $-cyclic operation on $\pi$ and
obtain a new permutation $\theta$. We have
\[
\theta^{\ast}=(+0, +2, +5, +1, -3, +4, +6)\cdot(-6, -4, +3, -1, -5, -2,
-0).
\]
Hence, $\theta=\langle+6\;\;+4\;\;-3\;\;+1\;\;+5\;\;+2\rangle$.
\end{itemize}
\end{example}

\begin{definition}
\label{x-x+1} For $\pi\in B_{n}$ such that $\pi^{\ast}\in S(2+2n)$
contains either
\[
|x|\rightarrow-|x|+1, ~~~~ |x|+1\rightarrow-|x|,
\]
or
\[
-|x|\rightarrow|x|+1, ~~~~ -|x|+1\rightarrow|x|,
\]
we define $\left(  [|x|]--[|x|+1]\right)  $-sign-change operation by replacing
in $\pi^{\ast}$ the following replacements:

\begin{itemize}
\item Replace $x$ by $-x$ and $-x$ by $x$;

\item Replace $x+1$ by $-x-1$ and $-x-1$ by $x+1$.
\end{itemize}

i.e., If $\pi^{\ast}$ contains
\[
w\rightarrow|x|\rightarrow-(|x|+1)\rightarrow z, ~~~~ -z\rightarrow
|x|+1\rightarrow-|x|\rightarrow-w,
\]
then by performing a $\left(  [|x|]--[|x|+1]\right)  $-sign-change operation on
$\pi$ we obtain $\theta$ such that $\theta^{\ast}$ contains
\[
w\rightarrow-|x|\rightarrow|x|+1\rightarrow z, ~~~~ -z\rightarrow
-(|x|+1)\rightarrow|x|\rightarrow-w.
\]
By performing a $\left(  [|x|]--[|x|+1]\right)  $-sign-change operation on
$\theta$ we obtain back $\pi$.
\end{definition}

\begin{proposition}
\label{obs.sign.change} Let $\pi$ and $\theta$ be two signed permutations of
$B_{n}$ such that $\theta$ is obtainable from $\pi$ by a $\left(
[|x|]--[|x|+1]\right)  $-sign-change operation. Then the following properties hold:

\begin{itemize}
\item If $|x|\neq0$ or $|x|\neq n$, then $m(\theta)=m(\pi)$;

\item $s(\theta)=s(\pi)$;

\item ${\Pr}_{\pi}(G)={\Pr}_{\theta}(G)$.
\end{itemize}
\end{proposition}

\begin{proof}
The proof of the first part comes directly from Definition \ref{x-x+1}. Thus,
we turn to the second part of the proposition. Assume $\pi^{\ast}$
contains
\[
w\rightarrow|x|\rightarrow-(|x|+1)\rightarrow z, ~~~~ -z\rightarrow
|x|+1\rightarrow-|x|\rightarrow-w,
\]
Therefore, $\Gr(\pi)$ has the form
\begin{align*}
&  \ldots\leftrightsquigarrow[w]\longleftrightarrow[-|x|]\leftrightsquigarrow
[|x|-1]\longleftrightarrow\ldots,\\
&  \ldots\leftrightsquigarrow[-z]\longleftrightarrow
[-(|x|+1)]\leftrightsquigarrow[|x|)]\longleftrightarrow
[|x|+1]\leftrightsquigarrow[-(|x|+2))]\longleftrightarrow\ldots .
\end{align*}
Now, we obtain $\theta$ by performing a $\left(  [|x|]--[|x|+1]\right)
$-sign-change operation such that $\theta^{\ast}$ contains
\[
w\rightarrow-|x|\rightarrow|x|+1\rightarrow z, ~~~~ -z\rightarrow
-(|x|+1)\rightarrow|x|\rightarrow-w,
\]
Therefore, $\Gr(\theta)$ has the form
\begin{align*}
&  \ldots\leftrightsquigarrow[w]\longleftrightarrow[|x|]\leftrightsquigarrow
[-(|x|+1)]\longleftrightarrow[-|x|]\leftrightsquigarrow
[|x|-1]\longleftrightarrow\ldots,\\
&  \ldots\leftrightsquigarrow[-z]\longleftrightarrow
[|x|+1]\leftrightsquigarrow[-(|x|+2))]\longleftrightarrow\ldots .
\end{align*}
Hence, $s(\pi)=s(\theta)$. Now, we prove the last part of the proposition,
i.e., ${\Pr}_{\pi}(G)={\Pr}_{\theta}(G)$, in case $\theta$ is obtainable from
$\pi$ by a $\left(  [|x|]--[|x|+1]\right)  $-sign-change operation. First,
assume $|x|\neq0$ and $|x|\neq n$. Assume also, $\pi_{k}=|x|+1$ and $\pi
_{k+1}=-|x|$ for some $+1\leq k\leq n-1$. Then
\begin{align*}
{\Pr}_{\pi}(G)  &  =\Pr(a_{1}\cdots a_{x-1}a_{x}(a_{x+1}a_{x}^{-1}%
)a_{x}a_{x+2}\cdots a_{n}=a_{|\pi|_{1}}^{\epsilon_{\pi_{1}}}\cdots
a_{|\pi|_{k-1}}^{\epsilon_{\pi_{k-1}}}(a_{x+1}a_{x}^{-1})a_{|\pi|_{k+1}%
}^{\epsilon_{\pi_{k+1}}}\cdots a_{|\pi|_{n}}^{\epsilon_{\pi_{n}}})\\
&  = \Pr(a_{1}\cdots a_{x-1}a_{x+1}(a_{x+1}^{-1}a_{x})a_{x+1}a_{x+2}\cdots
a_{n}=a_{|\pi|_{1}}^{\epsilon_{\pi_{1}}}\cdots a_{|\pi|_{k-1}}^{\epsilon
_{\pi_{k-1}}}(a_{x+1}^{-1}a_{x})a_{|\pi|_{k+1}}^{\epsilon_{\pi_{k+1}}}\cdots
a_{|\pi|_{n}}^{\epsilon_{\pi_{n}}})\\
&  = \Pr(a_{1}\cdots a_{n}=a_{|\theta|_{1}}^{\epsilon_{\theta_{1}}}\cdots
a_{|\theta|_{n}}^{\epsilon_{\theta_{n}}})\\
&  = {\Pr}_{\theta}(G).
\end{align*}
The case, $|x|\neq0$, ~$|x|\neq n$, ~$\pi_{k}=|x|$ and $\pi_{k+1}=-(|x|+1)$,
can be proved by the same argument, while showing
\[
\Pr\left((a_{1}\cdots a_{n})^{-1}=(a_{|\pi|_{1}}^{\epsilon_{\pi_{1}}}\cdots
a_{|\pi|_{n}}^{\epsilon_{\pi_{n}}})^{-1}\right)=\Pr\left((a_{1}\cdots a_{n})^{-1}%
=(a_{|\theta|_{1}}^{\epsilon_{\theta_{1}}}\cdots a_{|\theta|_{n}}%
^{\epsilon_{\theta_{n}}})^{-1}\right).
\]
Now, it remains to show ${\Pr}_{\pi}(G)={\Pr}_{\theta}(G)$, in case $\theta$ is
obtainable from $\pi$ by a \newline$\left(  [|x|]--[|x|+1]\right)
$-sign-change operation, where $|x|=0$ or $|x|=n$. Assume, $|x|=n$ (The case
of $|x|=0$ is proved by the same argument). Consider the following notations:

\begin{itemize}
\item $\alpha=a_{1}a_{2}\cdots a_{n-1}$;

\item $\beta=a_{|\pi|_{1}}^{\epsilon_{\pi_{1}}}\cdot a_{|\pi|_{2}}%
^{\epsilon_{\pi_{2}}}\cdots a_{|\pi|_{n-1}}^{\epsilon_{\pi_{n-1}}}$.
\end{itemize}

Then
\[
{\Pr}_{\pi}(G)=\Pr(\alpha\cdot a_{n}=\beta\cdot a_{n}^{-1})=\Pr(a_{n}%
^{2}=\alpha^{-1}\cdot\beta),
\]
where $a_{n}$ is independent on $\alpha$ and on $\beta$. Since we obtain
$\theta$ from $\pi$ by a $\left(  [n]--[0]\right)  $-sign-change operation, we
have
\[
{\Pr}_{\theta}(G)=\Pr\left(\alpha\cdot a_{n}=a_{n}^{-1}\cdot\beta^{-1}\right)=\Pr
\left((a_{n}\cdot\alpha)^{-2}=\beta\cdot\alpha^{-1}\right)=\Pr\left(\gamma^{2}=\alpha
^{-1}\cdot\beta\right),
\]
where $\gamma=\beta^{-1}\cdot a_{n}\cdot\alpha\cdot\beta$. Since, $a_{n}$ is
independent on $\alpha$ and on $\beta$, we have $\gamma$ is independent on
$\alpha$ and on $\beta$ as well. Hence, ${\Pr}_{\pi}(G)={\Pr}_{\theta}(G)$.
\end{proof}

\begin{example}
Consider the following signed permutation $\pi\in B_{n}$

\begin{itemize}
\item $\pi=\langle-6\;\;-2\;\;+3\;\;+1\;\;+5\;\;+4\rangle$. Then
\[
\pi^{\ast}=(+0, +4, +5, +1, +3, -2, -6)\cdot(+6, +2, -3, -1, -5, -4, -0).
\]
We can perform $\left(  [2]--[3]\right)  $-sign-change operation on $\pi$ and
obtain a new permutation $\theta$. We have
\[
\theta^{\ast}=(+0, +4, +5, +1, -3, +2, -6)\cdot(+6, -2, +3, -1, -5, -4,
-0).
\]
Hence, $\theta=\langle-6\;\;+2\;\;-3\;\;+1\;\;+5\;\;+4\rangle$;

\item $\pi=\langle+3\;\;-4\;\;+5\;\;+2\;\;-1\;\;-6\rangle$. Then
\[
\pi^{\ast}=(+0, -6, -1, +2, +5, -4, +3)\cdot(-3, +4, -5, -2, +1, +6, -0).
\]
We can perform $\left(  [6]--[0]\right)  $-sign-change operation on $\pi$ and
obtain a new permutation $\theta$. We have
\[
\theta^{\ast}=(-0, +6, -1, +2, +5, -4, +3)\cdot(-3, +4, -5, -2, +1, -6,
+0).
\]
Hence, $\theta=\langle-6\;\;+1\;\;-2\;\;-5\;\;+4\;\;-3\rangle$.
\end{itemize}
\end{example}

\begin{definition}
\label{x-y-equiv} Two signed permutations $\pi$ and $\theta$ in $B_{n}$ are
considered $\left(  [.]--[.]\right)  $-equivalent, if it is possible to obtain
either $\pi$ from $\theta$ or $\theta$ from $\pi$ by a finite sequence of
$\left(  [x]--[y]\right)  $-exchange, ~$\left(  [x]--[y]\right)  $-cyclic, or
$\left(  [|x|]--[|x|+1]\right)  $-sign-change operations.
\end{definition}

\begin{proposition}
\label{m_equal_zero} Let $\theta$ and $\pi$ be two signed permutations in $B_{n}$, which
are $\left(  [.]--[.]\right)  $-equivalent, then $m(\pi)=0$ if and only if
$m(\theta)=0$.
\end{proposition}

\begin{proof}
Let $\theta$ and $\pi$ be two signed permutations in $B_{n}$, which are $\left(
[.]--[.]\right)  $-equivalent. Then either $\theta$ is obtainable from $\pi$
or $\pi$ is obtainable from $\theta$ by a sequence of $\left(
[x]--[y]\right)  $-exchange, $\left(  [x]--[y]\right)  $-cyclic, and $\left(
[|x|]--[|x|+1]\right)  $-sign-change operations. Therefore, it is enough to
prove the statement when $\theta$ is obtainable from $\pi$ either by one $\left(
[x]--[y]\right)  $-exchange, or by one $\left(  [x]--[y]\right)  $-cyclic, or
by one $\left(  [|x|]--[|x|+1]\right)  $-sign-change operation. Then by
Propositions \ref{obs.exchange}, \ref{obs.cyclic}, \ref{obs.sign.change}, we
conclude $m(\theta)=0$ if and only if $m(\pi)=0$.
\end{proof}

\begin{lemma}
\label{equiv-hultman} Let $\pi$ and $\theta$ in $B_{n}$, which are $\left(
[.]--[.]\right)  $-equivalent, then

\begin{itemize}
\item $s(\pi)=s(\theta)$;

\item ${\Pr}_{\pi}(G)={\Pr}_{\theta}(G)$ in every finite group $G$.
\end{itemize}
\end{lemma}

\begin{proof}
By Propositions \ref{obs.exchange}, \ref{obs.cyclic}, and
\ref{obs.sign.change} $s(\pi)=s(\theta)$ and ${\Pr}_{\pi}(G)={\Pr}_{\theta
}(G)$ if $\theta$ is obtainable from $\pi$ either by one $\left([x]--[y]\right)$-exchange, or by one $\left([x]--[y]\right)$-cyclic, or
by one $\left([|x|]--[|x|+1]\right)$-sign-change operation. Therefore, the
results hold in case, where $\theta$ is obtainable from $\pi$ by any finite
number of $\left(  [x]--[y]\right)  $-exchange, or ~$\left([x]--[y]\right)$-cyclic or $\left([|x|]--[|x|+1]\right)$-sign-change operations.
\end{proof}

\section{The main result}

\label{main}

In this section, we prove the main result of the paper, whereby Theorem
\ref{main-theorem}, we show the connection between $\Pr_{\pi}(G)$ and
$s(\pi)$, the number of the alternating cycles of $\pi$ in the breakpoint
graph $\Gr(\pi)$ for every $\pi\in B_{n}$. Theorem \ref{main-theorem} is a
generalization of the main theorem of \cite{CGLS}, where it has been proved
that for $\pi\in S_{n}$, the generalized commuting probability, $\Pr_{\pi}(G)$
depends on the number of the alternating cycles in the cycle graph of $\pi$.
The proof of Theorem \ref{main-theorem} makes use of several technical lemmas.

\begin{lemma}
\label{add-n} Let $\pi^{\prime}\in B_{n-1}$ be a signed permutation, and let
$\pi$ be a signed permutation in $B_{n}$ such that

\begin{itemize}
\item $\pi_{i}=\pi^{\prime}_{i}$ for every $i$ such that $|i|\leq n-1$;

\item $\pi_{+n}=-n$.
\end{itemize}

Then $s(\pi)=s(\pi^{\prime})$.
\end{lemma}

\begin{proof}
Look at the arrows of $\Gr(\pi)$ and $\Gr(\pi^{\prime})$. Since $\pi_{i}%
=\pi^{\prime}_{i}$ for every $i$ such that $|i|\leq n-1$, all the arrows of
$\Gr(\pi)$ and $\Gr(\pi^{\prime})$ are the same, apart from the following segment:

\begin{itemize}
\item In $\Gr(\pi^{\prime})$ the following holds:
\[
\ldots\leftrightsquigarrow[\pi_{+(n-1)}]\longleftrightarrow
[-0]\leftrightsquigarrow[+(n-1)]\longleftrightarrow\ldots .
\]

\item In $\Gr(\pi)$ the following holds:
\[
\ldots\leftrightsquigarrow[\pi_{+(n-1)}]\longleftrightarrow
[+n]\leftrightsquigarrow[-0]\longleftrightarrow[-n]\leftrightsquigarrow
[+(n-1)]\longleftrightarrow\ldots .
\]

\end{itemize}

Hence, $s(\pi)=s(\pi^{\prime})$.
\end{proof}

\begin{lemma}
\label{pi-pi} Let $\pi^{\prime}$ and $\theta^{\prime}$ be two $\left(
[x]--[y]\right)  $-equivalent signed permutations in $B_{n-1}$. Let $\pi$ and
$\theta$ be two signed permutations in $B_{n}$ such that

\begin{itemize}
\item $\pi^{\prime}_{i}=\pi_{i}$ and $\theta^{\prime}_{i}=\theta_{i}$ for
every $i$ such that $|i|\leq n-1$;

\item $\pi_{+n}=\theta_{+n}=-n$.
\end{itemize}

Then $\pi$ and $\theta$ are $\left(  [.]--[.]\right)  $-equivalent as well.
\end{lemma}

\begin{proof}
It is enough to show that $\pi$ and $\theta$ are $\left(  [x]--[y]\right)
$-equivalent in case $\theta^{\prime}$ is obtainable by either performing a single
$\left(  [x]--[y]\right)  $-exchange, or a single $\left(  [x]--[y]\right)
$-cyclic, or a single $\left(  [|x|]--[|x|+1]\right)  $-sign-change operation
on $\pi^{\prime}$. If $\theta^{\prime}$ is obtainable by performing a $\left(
[x]--[y]\right)  $-exchange operation on $\pi^{\prime}$ for $x\neq+(n-1)$,
then we obtain $\theta$ by performing a $\left(  [x]--[y]\right)  $-exchange
operation on $\pi$ for the same $x$ and $y$. Therefore, assume $\theta
^{\prime}$ is obtainable by performing a $\left(  [+(n-1)]--[y]\right)  $-exchange
operation on $\pi^{\prime}$. Then by Proposition \ref{obs.exchange},
$y=\pi^{\prime}_{+(n-1)}$ and either $\pi^{\prime}_{i}=+(n-1)$ and then
\[
\theta^{\prime}=\langle\pi_{+1}\;\;\cdots\;\;+(n-1)\;\;\pi_{+(n-1)}%
\;\;\pi_{i+1}\;\;\cdots\;\;\pi_{+(n-2)}\rangle
\]
or $\pi_{i}=-(n-1)$ and then
\[
\theta^{\prime}=\langle\pi_{+1}\;\;\cdots\;\;\pi_{-(n-1)}\;\;-(n-1)\;\;\pi
_{i+1}\;\;\cdots\;\;\pi_{+(n-2)}\rangle.
\]
Now, by performing a $\left(  [-0]--[\pi_{+(n-1)}]\right)  $-exchange operation
on $\pi$, we obtain
\[
\tau=\langle\pi_{+1}\;\;\cdots\;\;\pi_{+(n-2)}\;\;-n\;\;\pi_{-(n-1)}\rangle.
\]
Then by performing a $\left(  [+(n-1)]--[\pi_{+(n-1)}]\right)  $-exchange
operation on $\tau$, we obtain $\theta$. Now assume, $\theta^{\prime}$ is
obtainable by performing a $\left(  [x]--[y]\right)  $-cyclic operation on $\pi^{\prime}%
$. If $x\neq n-1$ or $x\neq-(n-1)$, then one can obtain $\theta$ by performing a $\left(
[x]--[y]\right)  $-cyclic operation on $\pi$ for the same $x$ and $y$.
Therefore, assume either $x=n-1$ or \newline$x=-(n-1)$. If $x=n+1$ and
${\pi^{\prime}}^{\ast}_{0}=y$, then
\[
\pi^{\ast}=(+0, -n, y, n-1, \ldots)\cdot(\ldots, -n+1, -y, +n, -0).
\]
In case $\sign\left(y\right)  =(+)$, we obtain theta by performing first
a $\left(  [+(n-1)]--[y]\right)  $-cyclic, then  a $\left([y]--[+n]\right)
$-exchange, and finally a $\left(  [+0]--[-y]\right)  $-exchange
operation.\newline In case $\sign\left(y\right)=(-)$, we obtain theta by
performing first $\left(  [+(n-1)]--[y]\right)  $-cyclic, then $\left(
[y-1]--[+n]\right)  $-exchange, and finally $\left(  [+n]--[-y+1]\right)
$-exchange operation. By similar arguments, $\theta$ is obtainable from $\pi$
in all the rest of the cases, where $\theta^{\prime}$ is obtainable from
$\pi^{\prime}$ by performing a $\left(  [x]--[y]\right)  $-cyclic operation. Now assume,
$\theta^{\prime}$ is obtainable by a $\left(  [|x|]--[|x|+1]\right)
$-sign-change operation on $\pi^{\prime}$. If $|x|\neq n-1$, then one can
obtain $\theta$ by performing a $\left(  [|x|]--[|x|+1]\right)  $-sign-change operation
on $\pi$ for the same $|x|$. Therefore, assume $|x|=n-1$. Then we have
\[
\pi^{\ast}=(+0, -n, -(n-1), \pi_{+(n-2)}, \ldots, \pi_{+1})\cdot(\pi_{-1},
\ldots, \pi_{-(n-2)}, +(n-1), +n, -0).
\]
By performing a $\left(  [n]--[0]\right)  $-sign-change operation on $\pi$, we
obtain $\mu$ such that
\[
\mu^{\ast}=(+0, \pi_{-1}, \ldots, \pi_{-(n-2)}, +(n-1), -n)\cdot(+n,
-(n-1), \pi_{+(n-2)}, \ldots, \pi_{+1}, -0).
\]
By performing a $\left(  [n-1]--[n]\right)  $-sign-change operation on $\mu$, we
obtain $\eta$ such that
\[
\eta^{\ast}=(+0, \pi_{-1}, \ldots, \pi_{-(n-2)}, -(n-1), +n)\cdot(-n,
+(n-1), \pi_{+(n-2)}, \ldots, \pi_{+1}, -0).
\]
Finally, by performing a $\left(  [+(n-1)]--[-0]\right)  $-cyclic operation on
$\eta$, we obtain the desired $\theta$, where:
\[
\theta^{\ast}=(+0, -n, \pi_{-1}, \ldots, \pi_{-(n-2)}, -(n-1))\cdot(+(n-1),
\pi_{+(n-2)}, \ldots, \pi_{+1}, +n, -0).
\]

\end{proof}

\begin{lemma}
\label{0-n} Let $\pi\in B_{n}$ be a signed permutation such that $\pi_{i}=-n$,
for some \newline$1\leq i\leq n$, then there exists $\theta\in B_{n}$ such
that the following holds:

\begin{itemize}
\item $\theta_{+n}=-n$;

\item $\pi$ is obtainable by a sequence of $n-i$ consecutive $\left(
[-0]--[y]\right)  $-exchange operations starting on $\theta$, such
that $\theta_{j}=\pi_{j}$ for every $j$ such that $+1\leq j<i$;

\item $\theta$ and $\pi$ are $\left(  [.]--[.]\right)  $-equivalent;

\item $s(\theta)=s(\pi)$.
\end{itemize}
\end{lemma}

\begin{proof}
Let $\pi\in B_{n}$ be a signed permutation such that $\pi_{i}=-n$ for some
$+2\leq i\leq+n$. Then
\[
\pi^{\ast}=(+0, \pi_{+n}, \ldots, \pi_{i+1}, -n, \pi_{i-1}, \ldots,
\pi_{+1})\cdot(\pi_{-1}, \ldots, \pi_{-(i-1)}, +n, \pi_{-(i+1)}, \ldots,
\pi_{-n}, -0).
\]
Now, by performing a $\left(  [-0]--[\pi_{i-1}]\right)  $-exchange operation, we
obtain $\mu$ such that
\begin{align*}
\mu^{\ast} & =(+0, \pi_{-(i-1)}, \pi_{+n}, \ldots, \pi_{i+1}, -n, \pi
_{i-2}, \ldots, \pi_{+1})\cdot\\
& (\pi_{-1}, \ldots, \pi_{-(i-2)}, +n, \pi_{-(i+1)}, \ldots, \pi_{-n},
\pi_{i-1}, -0).
\end{align*}
Therefore, we obtain $\mu$ such that

\begin{itemize}
\item $\mu_{i-1}=\pi_{i}=-n$;

\item $\mu_{j}=\pi_{j}$ for $j$ such that $+1\leq j<i-1$;

\item $\mu_{j}=\pi_{j+1}$ for $j$ such that $i-1\leq j\leq+(n-1)$;

\item $\theta_{+n}=\pi_{-(i-1)}$.
\end{itemize}

Therefore, we conclude that every $\pi\in B_{n}$ such that $\pi_{i}=-n$ for
some $+1\leq i\leq+n$, is obtainable from $\theta$ such that $\theta_{+n}=-n$
by performing $\left(  [-0]--[y]\right)  $-exchange operations $n-i$ times
repeatedly, starting on $\theta$ such that $\pi_{j}=\theta_{j}$ for every $j$ such that
$+1\leq j<i$. Hence, we get the rest of the results of the lemma.
\end{proof}

\begin{lemma}
\label{0-n-equiv} Let $\pi\in B_{n}$ be a signed permutation such that
$\sign(\pi_{i})=(+)$ and $\pi_{i-1}=\pi_{-i}+1$ for some $+2\leq i\leq+n$, and
there exists $j$ such that $+1\leq j<i$ and $\pi_{j}=+n$. Then there exists a signed
permutation $\theta\in B_{n}$ such that the following holds:

\begin{itemize}
\item $\theta_{+n}=-n$;

\item $\theta_{k}=\pi_{i-1}$ for some $+1\leq k\leq+(n-1)$;

\item $\theta$ and $\pi$ are $\left(  [.]--[.]\right)  $-equivalent;

\item $s(\theta)=s(\pi)$.
\end{itemize}
\end{lemma}

\begin{proof}
Assume $\pi\in B_{n}$ is a signed permutation such that $\sign(\pi_{i})=(+)$ and
$\pi_{i-1}=\pi_{-i}+1$ for some $+2\leq i\leq+n$, and there exists $j$ such that $+1\leq
j<i-1$ and $\pi_{j}=+n$. Then by performing $\left(  [-\pi
_{i}]--[y]\right)  $-exchange operations $i-1-j$ times repeatedly, where
starting on $\pi$, we obtain a signed permutation $\mu$ such that

\begin{itemize}
\item $\mu_{i}=\pi_{i}$;

\item $\mu_{i-1}=-n$;

\item $\mu_{j}=\pi_{i-1}=\pi_{-i}+1$, which implies $\sign(\mu_{j})=(-)$, since
$\sign(\pi_{-i})=(-)$.
\end{itemize}

Now, by using Lemma \ref{0-n}, we conclude that $\mu$ is obtainable from
$\theta$ by a sequence of $\left(  [-0]--[y]\right)  $-exchange operations
$n-i-1$ times repeatedly such that

\begin{itemize}
\item $\theta_{+n}=-n$;

\item $\theta_{k}=\mu_{k}$ for every $k$ such that $+1\leq k<i-1$.
\end{itemize}

Since $j<i-1$, we have $\theta_{j}=\mu_{j}=\pi_{-i}+1$. Therefore,
$\sign(\theta_{j})=(-)$. Now, since $\theta$ and $\mu$ are $\left(
[.]--[.]\right)  $-equivalent and $\mu$ and $\pi$ are $\left(
[.]--[.]\right)  $-equivalent too, we have $\theta$ and $\pi$ are $\left(
[.]--[.]\right)  $-equivalent as well. Hence, the lemma holds.
\end{proof}

\begin{lemma}
\label{neg} Let $\pi\in B_{n}$ be a signed permutation such that $m(\pi)\geq
1$, then there exists $\theta\in B_{n}$ such that

\begin{itemize}
\item $\theta_{+n}=-n$;

\item $\theta$ and $\pi$ are $\left(  [.]--[.]\right)  $-equivalent;

\item $s(\theta)=s(\pi)$.
\end{itemize}
\end{lemma}

\begin{proof}
The proof is in induction on $n$. Notice, ~$\pi=\langle-1\rangle$ ~is the only
element in $B_{1}$ with $m(\pi)\geq1$, therefore the lemma holds trivially in
case of $n=1$. Assume by induction that the lemma holds for $n_{0}<n$, and we
prove it for $n_{0}=n$. If $\pi_{i}=-n$ for some $+0\leq i\leq+n$, then by
Lemma \ref{0-n}, $\pi$ is $\left(  [.]--[.]\right)  $-equivalent to a signed
permutation $\theta\in B_{n}$ such that $\theta_{+n}=-n$ and $s(\theta
)=s(\pi)$. Moreover, by Lemma \ref{pi-pi}, $s(\theta)=s(\theta^{\prime})$,
where $\theta^{\prime}\in B_{n-1}$ such that $\theta^{\prime}_{i}=\theta_{i}$
for every $-(n-1)\leq i\leq+(n-1)$. If $\sign(\theta^{\prime}_{j})=(-)$ for
some $+1\leq j\leq+(n-1)$, then by the induction hypothesis, $\theta^{\prime}$
is $\left(  [.]--[.]\right)  $-equivalent to every $\tau^{\prime}$ such that
$s(\tau^{\prime})=s(\theta^{\prime})$ and $\sign(\tau^{\prime}_{k})=(-)$ for
some $+1\leq k\leq+(n-1)$. Then by Lemma \ref{pi-pi}, $s(\tau)=s(\tau^{\prime
})$, where $\tau\in B_{n}$ such that $\tau_{i}=\tau^{\prime}(i)$ for every
$-(n-1)\leq i\leq+(n-1)$ and $\tau_{+n}=-n$. Since $\theta^{\prime}$ is
$\left(  [.]--[.]\right)  $-equivalent to $\tau^{\prime}$, by Lemma
\ref{add-n}, $\theta$ is $\left(  [.]--[.]\right)  $-equivalent to $\tau$ as
well. Thus we conclude, $\theta$ is $\left(  [.]--[.]\right)  $-equivalent to
$\tau$ if the following holds:

\begin{itemize}
\item $\theta_{+n}=\tau_{+n}=-n$;

\item $s(\theta)=s(\tau)$;

\item $\sign(\theta_{i})=\sign(\tau_{j})=(-)$ for some $+1\leq i, j\leq+(n-1)$.
\end{itemize}

Now, consider $\tau\in B_{n}$ such that $\tau_{+n}=-n$ but $\sign(\tau
_{j})=(+)$ for all $+1\leq j\leq+(n-1)$. Then
\[
\tau^{\ast}=(+0, -n, \tau_{+(n-1)}, \ldots, \tau_{+1})\cdot(\tau_{-1},
\ldots, \tau_{-(n-1)}, +n, -0).
\]
Now, by performing a $\left(  [-n]--[\tau_{+(n-1)}]\right)  $-exchange operation
on $\tau$ we obtain $\mu$ such that
\[
\mu^{\ast}=(+0, \tau_{-(n-1)}, -n, \ldots, \tau_{+1})\cdot(\tau_{-1},
\ldots, +n, \tau_{+(n-1)}, -0).
\]
Now, by performing a $\left(  [-n]--[\tau_{-(n-1)}]\right)  $-cyclic operation
on $\mu$, we obtain $\zeta$ such that
\[
\zeta^{\ast}=(+0, \tau_{-(n-1)}-1, \tau_{+(n-1)}, \ldots, +n, \ldots
)\cdot(\ldots, -n, \ldots, \tau_{-(n-1)}, \tau_{+(n-1)}+1, -0).
\]
Now, by performing a $\left(  [\tau_{+(n-1)}]--[\tau_{+(n-1)}+1]\right)
$-sign-change operation on $\zeta$, we obtain $\eta$ such that
\[
\eta^{\ast}=(+0, \tau_{+(n-1)}+1, \tau_{-(n-1)}, \ldots, +n, \ldots
)\cdot(\ldots, -n, \ldots, \tau_{+(n-1)}, \tau_{-(n-1)}-1, -0).
\]
Since $\sign(\eta_{+n})=(+)$, $\eta_{n-1}=\eta_{-n}+1$, and $\eta_{j}=+n$ for
some $+1\leq j<+n$, by Lemma \ref{0-n-equiv}, there exists a signed
permutation $\theta\in B_{n}$ such that

\begin{itemize}
\item $\theta_{+n}=-n$;

\item $\theta_{k}=\eta_{+(n-1)}$ for some $+1\leq k\leq+(n-1)$;

\item $\theta$ and $\eta$ are $\left(  [.]--[.]\right)  $-equivalent;

\item $s(\theta)=s(\eta)$.
\end{itemize}

Now, since $\eta$ and $\pi$ are $\left(  [.]--[.]\right)  $-equivalent and
$s(\eta)=s(\pi)$, we conclude $\pi$ and $\theta$ are $\left(  [x]--[y]\right)
$-equivalent and $s(\pi)=s(\theta)$ as well.
\end{proof}

\begin{theorem}
\label{main-theorem} Let $\pi$ and $\theta$ be two signed permutations in
$B_{n}$ such that $s(\pi)=s(\theta)=k$, then

\begin{itemize}
\item  if $m(\pi)=m(\theta)=0$, then $\pi$ and $\theta$ are $\left(
[x]--[y]\right)  $-equivalent, and
\[
{\Pr}_{\pi}(G)={\Pr}_{\theta}(G)={\Pr}^{n-k+1}(G).
\]

\item if $m(\pi)>0$ and $m(\theta)>0$, then $\pi$ and $\theta$ are
$\left(  [.]--[.]\right)  $-equivalent, and
\[
{\Pr}_{\pi}(G)={\Pr}_{\theta}(G)={\Pr}^{-(n-k+1)}(G).
\]

\end{itemize}
\end{theorem}

\begin{proof}
If $m(\pi)=0$, then $\pi=|\pi|$, and therefore ${\Pr}_{\pi}(G)={\Pr}_{|\pi
|}(G)$. Then we deal with the special case, which have been proved in
\cite{CGLS}. Therefore, assume $m(\pi)>0$. Then by Lemma \ref{neg}, $\pi$ is
$\left(  [.]--[.]\right)  $-equivalent to a signed permutation $\theta$ such
that $s(\theta)=s(\pi)$ and $\theta_{+n}=-n$. By Proposition \ref{i-k},
$s(I^{(k)})=s(\pi)=s(\theta)$ for $k=n-s(\pi)+1$. Therefore, by using Lemma
\ref{neg} again, $I^{(k)}$ is $\left(  [.]--[.]\right)  $-equivalent to
$\theta$ as well. Hence, we conclude that $I^{(k)}$ and $\pi$ are $\left(
[x]--[y]\right)  $-equivalent for $k=n-s(\pi)+1$. Hence, we get
\[
{\Pr}_{\pi}(G)={\Pr}_{\theta}(G)={\Pr}^{-(n-k+1)}(G).
\]

\end{proof}

We notice that there are two extreme cases of finite group $G$ for applying
Theorem \ref{main-theorem}.

\begin{enumerate}
\item The case of finite ambivalent group $G$, where every element $g\in G$ is
conjugate to its inverse $g^{-1}\in G$.

\item The case of group of odd order $G$, where no element $g\neq1$ is
conjugate to its inverse $g^{-1}\in G$.
\end{enumerate}

Hence, we get the following two corollaries.

\begin{corollary}
\label{ambi-prob} Let $G$ be a finite group, then
\[
{\Pr}_{\theta}(G)={\Pr}_{\pi}(G),
\]
for every $\theta$ and $\pi$ in $B_{n}$ such that $s(\theta)=s(\pi)$ if and
only if $G$ is an ambivalent finite group. Which means, ${\Pr}_{\pi}(G)$
depends only on $s(\pi)$ (the number of alternating cycles in $\Gr(\pi)$),
regardless of whether $\pi$ is a positive or a non-positive signed permutation.
\end{corollary}

\begin{proof}
The result holds by applying Corollary \ref{ambivalent} in Theorem
\ref{main-theorem}
\end{proof}

\begin{corollary}
\label{odd-prob} Let $G$ be a finite group, then
\[
{\Pr}_{\theta}(G)={\Pr}_{\pi}(G)=\frac{1}{|G|},
\]
for every  $\theta$ and $\pi$ in $B_{n}$ such that $\theta$ and $\pi$ are
non-positive (which means $m(\theta)>0$ and $m(\pi)>0$), without any
dependance on the values of $s(\theta)$ and $s(\pi)$, if and only if $G$ is an
odd order group. That is every non-positive $\pi$ satisfies ${\Pr}_{\pi
}(G)=\frac{1}{|G|}$, regardless of the number of alternating cycles in
$\Gr(\pi)$.
\end{corollary}

\begin{proof}
The result holds by applying Corollary \ref{odd-order} in Theorem
\ref{main-theorem}
\end{proof}

Finally, we have the following corollary, which generalizes Corollary
\ref{odd-prob}.

\begin{corollary}
Let $G=G_{1}\bigoplus G_{2}$ such that $G_{1}$ is an abelian $2$-group, and
$G_{2}$ is a group, which has an odd order (i.e., $G_{1}$ is the $2$-sylow
subgroup of $G$). Then
\[
{\Pr}_{\theta}(G)={\Pr}_{\pi}(G)={\Pr}^{-1}(G)=\frac{\inv(G)}{|G|},
\]
for every $\theta$ and $\pi$ in $B_{n}$ such that $\theta$ and $\pi$ are
non-positive (which means $m(\theta)>0$ and $m(\pi)>0$).
\end{corollary}

\begin{proof}
Since $G_{1}$ is an abelian group, by Corollary \ref{abelian}, ${\Pr
}^{-k}(G_{1})=\frac{\inv(G_{1})}{|G_{1}|}$ for every $k\in\mathbb{N}$.
Similarly, since the order of $G_{2}$ is odd, by Corollary \ref{odd-order},
${\Pr}^{-k}(G_{2})=\frac{1}{|G_{2}|}$ for every $k\in\mathbb{N}$. Hence, by
using $G=G_{1}\bigoplus G_{2}$ and Corollary \ref{direct-sum}, we conclude
\[
{\Pr}^{-k}(G)={\Pr}^{-k}(G_{1})\cdot{\Pr}^{-k}(G_{2})=\frac{\inv(G_{1})}%
{|G_{1}|\cdot|G_{2}|}=\frac{\inv(G)}{|G|},
\]
for every $k\in\mathbb{N}$. Now, by applying Theorem \ref{main-theorem}, we
get the desired result of the corollary.
\end{proof}

\section{Conclusion and future plans}

\label{conc}

In this paper, we generalize the results of \cite{CGLS}, where we find an
interesting connection between the number of cycles in the breakpoint graph $\Gr(\pi)$ for a
signed permutation $\pi\in B_{n}$ and the signed generalized commuting
probability ${\Pr}_{\pi}(G)$, which is a generalization of the generalized
commuting probability, which was defined in \cite{CGLS}. In contrast to the
results of \cite{CGLS}, in the case of $\pi\in B_{n}$, the following changes occur:

\begin{itemize}
\item For $\pi\in B_{n}$, the parity of ~$n-s(\pi)+1$ ~is not necessarily
even, it can be any integer;

\item If $G$ is a non-ambivalent group, then ~$s(\pi)=2k$ ~induces two
equivalence classes of ~${\Pr}_{\pi}(G)$, namely:

\begin{itemize}
\item ${\Pr}^{2k}(G)={\Pr}_{\pi}(G)$ ~for a positive $\pi\in B_{n}$;

\item ${\Pr}^{-2k}(G)={\Pr}_{\pi}(G)$ ~for a non-positive $\pi\in B_{n}$,
\end{itemize}

such that ${\Pr}^{2k}(G)\neq{\Pr}^{-2k}(G)$;

\item If $G$ has an odd order, then ~${\Pr}_{\pi}(G)={\Pr}^{-k}(G)=\frac
{1}{|G|}$ ~for every non-positive $\pi\in B_{n}$ and every $k\in\mathbb{N}$.

\item If $G$ has an abelian $2$-sylow subgroup, and $G$ is a direct sum of
its $2$-sylow subgroup with an odd order group, then ~${\Pr}_{\pi}(G)={\Pr
}^{-k}(G)=\frac{\inv(G)}{|G|}$ ~for every non-positive $\pi\in B_{n}$ and every
$k\in\mathbb{N}$.
\end{itemize}

It might be interesting to find further generalizations of the generalized
commuting probability. For instance, classifying the probabilities of
\[
\Pr(a_{1}a_{2}\cdots a_{n}=a_{\pi_{1}}^{\prime}a_{\pi_{2}}^{\prime}\cdots
a_{\pi_{n}}^{\prime}),
\]
where $\pi$ is a permutation of $S_{n}$ and $a_{i}^{\prime}$ is a specific
automorphic or anti-automorphic image of $a_{i}$, under a defined automorphism
of the group $G$.


\begin{thebibliography}{99}                                                                                               %


\bibitem {AR}R. M. Adin, Y. Roichman, The Flag Major Index and Group Actions
on Polynomial Rings, \emph{Europ. J. Combinatorics} \textbf{22}, (2001), 431-446.

\bibitem {APA}N. Alexeev, A. Pologova, M. A. Alekseyev, Generalized Hultman
Numbers and Cycle Stuctures of Breakpoint Graphs, \emph{Journal of Computational Biology}, \textbf{24 (2)}, (2017), 93-105.

\bibitem {AST}M. Amram, R. Shwartz, M. Teicher, Coxeter covers of the
classical Coxeter groups, \emph{Interntionl Journal of Algebra and
Computation}, \textbf{20} (2010) 1041-1062.

\bibitem {BP}B. Bafna and P. A. Pevzner, Sorting by transpositions, \emph{SIAM
J. Discrete Math.} \textbf{11} (1998) 224-240.

\bibitem {BB}A. Bjorner, F. Brenti, Combinatorics of Coxeter groups, in: GTM,
vol. 231, Springer, 2004.

\bibitem {CGLS}Y. Cherniavsky, A. Goldstein, V. E. Levit, R. Shwartz, Hultman
Numbers and Generalized Commuting Probability in Finite Groups, \emph{Journal
of Integer Sequences} \textbf{20} (2017), Article 17.10.7.

\bibitem {DasNath1}A. K. Das and R. K. Nath, A generalization of commutativity
degree of finite groups, \emph{Communications in Algebra} \textbf{40} (2012) 1974-1981.

\bibitem {DasNath2}A. K. Das, R. K. Nath and M. R. Pournaki, A survey on the
estimation of commutativity in finite groups, \emph{Southeast Asian
Bulletin of Mathematics}, \textbf{37} (2013), no. 2, 161-180.

\bibitem {DL}J. P. Doignon, A. Labarre, On Hultman Numbers, \emph{Journal of
Integer Sequences} \textbf{10} (2007), Article 07.6.2.

\bibitem {GL}S. Grusea, A. Labarre, The distribution of cycles in breakpoint
graphs of signed permutations, \emph{Discrete Applied Mathematics}
\textbf{161} (2013) 1448--1466.

\bibitem {G}W. H. Guftafson, What is the probability that two group elements
are commute?, \emph{American Mathematical Monthly} \textbf{80} (1973) 1031-1034.

\bibitem {Hultman1999}A. Hultman, Toric permutations, Master's thesis,
Department of Mathematics, KTH, Stockholm, Sweden (1999).

\bibitem {Mac}P. A. MacMahon, Combinatory Analysis I-II, \emph{Cambridge
University Press, London/New-York} (1916) (Reprinted by Chelsea, New-York 1960.).

\bibitem {Nath2011}R. K. Nath, Some new directions in commutativity degree of
finite groups: Recent developments, \emph{Lambert Academic Publishing} (2011).

\bibitem {NathDash1}R. K. Nath and A.K. Das, On generalized commutativity
degree of a finite group, \emph{Rocky Mountain Journal of Mathematics}
\textbf{41} (2011) 1987-2000.

\bibitem {rtv}L. Rowen, M. Teicher, U. Vishne, Coxeter covers of the symmetric
groups, \emph{Journal of Group Theory}, \textbf{8}, (2005) 139-169.

\bibitem {SAR}R. Shwartz, R. M. Adin, and Y. Roichman, Major Indices and
Perfect Bases for Complex Reflection Groups, \emph{The Electronic Journal of
Combinatorics} \textbf{15} (2008) Research paper 61.
\end{thebibliography}
\end{document}